\documentclass[3p]{elsarticle}

\usepackage{loreaux2}
\usepackage{dsfont}
\usepackage{enumitem}
\usepackage[centercolon]{mathtools}
\usepackage[colorlinks=true, linkcolor=black, citecolor=black]{hyperref}

\usepackage[usenames,dvipsnames]{xcolor}
\usepackage{graphicx}

\numberwithin{equation}{section}

\newtheorem{theorem}{Theorem}[section]
\newtheorem{lemma}[theorem]{Lemma}
\newtheorem{proposition}[theorem]{Proposition}
\newtheorem{corollary}[theorem]{Corollary}

\theoremstyle{definition}
\newtheorem{definition}[theorem]{Definition}
\newtheorem{remark}[theorem]{Remark}

\newtheorem{conjecture}[theorem]{Conjecture}
\newtheorem{example}[theorem]{Example}

\DeclareMathOperator{\trace}{Tr}
\DeclareMathOperator{\diag}{diag}
\DeclareMathOperator{\spans}{span}
\DeclareMathOperator{\convex}{co}
\DeclareMathOperator{\range}{ran}
\DeclareMathOperator{\dom}{dom}
\DeclareMathOperator{\rank}{rank}
\DeclareMathOperator{\supp}{supp}

\renewcommand{\today}{November 13, 2014}

\begin{document}

\title{Majorization and a Schur--Horn Theorem \\ for positive compact operators, \\ the nonzero kernel case}

\author[uc]{Jireh Loreaux\fnref{fn1}}
\ead{loreaujy@mail.uc.edu}
\ead{loreaujy@gmail.com}

\author[uc]{Gary Weiss\corref{cor}\fnref{fn2}}
\ead{gary.weiss@uc.edu}
\ead{weissg@ucmail.uc.edu}
\ead{gary.weiss@math.uc.edu}

\fntext[fn1]{Partially supported by funds from the Charles Phelps Taft Research Center.}
\fntext[fn2]{Partially supported by 
Simons Foundation Collaboration Grant for Mathematicians \#245014 and the Charles Phelps Taft Research Center.}
\cortext[cor]{Corresponding author}

\address[uc]{Department of Mathematical Sciences, 4199 French Hall West \\
University of Cincinnati,
2815 Commons Way \\
Cincinnati, OH 45221-0025, USA}


\begin{abstract}
Schur--Horn theorems focus on determining the diagonal sequences obtainable for an operator under all possible basis changes,
formally described as the range of the canonical conditional expectation of its unitary orbit.

Following a brief background survey, we prove an infinite dimensional Schur--Horn theorem for positive compact operators with infinite dimensional kernel, one of the two open cases posed recently by Kaftal--Weiss.
There, they characterized the diagonals of operators in the unitary orbits for finite rank or zero kernel positive compact operators.
Here we show how the characterization problem depends on the dimension of the kernel when it is finite or infinite dimensional.

We obtain exact majorization characterizations of the range of the canonical conditional expectation of the unitary orbits of positive compact operators with infinite dimensional kernel,
unlike the approximate characterizations of Arveson--Kadison, but extending the exact characterizations of Gohberg--Markus  and Kaftal--Weiss.

Recent advances in this subject and related subjects like traces on ideals show the relevance of new kinds of sequence majorization as in the work of Kaftal--Weiss (e.g., strong majorization and another majorization similar to what here we call $p$-majorization), and of Kalton--Sukochev (e.g., uniform Hardy--Littlewood majorization), and of Bownik--Jasper (e.g.,  Riemann and Lebesgue majorization).
Likewise key tools here are new kinds of majorization, which we call $p$- and  approximate $p$-majorization ($0\le p\le \infty$).
\end{abstract}

\maketitle

\section{Introduction}\label{sec:introduction}

The Schur--Horn Theorem in finite matrix theory characterizes the diagonals of a self-adjoint $n\times n$ matrix in terms of its eigenvalues.
In particular, if $\lambda_1,\ldots,\lambda_n$ are the eigenvalues of a self-adjoint matrix counting multiplicity, then its diagonal sequence $d_1,\ldots,d_n$ has the following relationship with its eigenvalues:
\[ \sum_{i=1}^m d^*_i \le \sum_{i=1}^m \lambda^*_i,\quad\text{for } 1\le m\le n,\text{ and}\quad \sum_{i=1}^n d_i = \sum_{i=1}^n \lambda_i, \]
where $d^*,\lambda^*$ are any monotone decreasing rearrangements of $d,\lambda$.
This relationship between the sequences $d,\lambda\in \mathbb{R}^n$ is called majorization and is historically denoted by $d\prec\lambda$.
Schur proved this diagonal-eigenvalue relationship in \cite{Sch-1923-SBMG} and Horn \cite{Hor-1954-AJoM} proved the converse.
That is, Horn proved that given $d\prec\lambda$, there exists a self-adjoint $n\times n$ matrix with eigenvalue  sequence $\lambda$ and diagonal sequence $d$. 

To modernize this perspective, let $H$ denote a Hilbert space (finite or separable infinite dimensional) and fix an orthonormal basis $\{e_n\}_{n=1}^N$ for $H$ ($1\le N\le \infty$).
 Denote by $\mathcal{D}$ the abelian algebra of diagonal operators  (the canonical atomic masa of $B(H)$) corresponding to the basis $\{e_n\}_{n=1}^N$ and $\mathcal{D}_{sa}$ the self-adjoint operators in $\mathcal{D}$.
Given an operator $X\in B(H)$, we denote by $E(X)$ the diagonal operator having as its diagonal the main diagonal of $X$ (i.e., $E:B(H)\to\mathcal{D}$ is the canonical faithful normal trace-preserving conditional expectation).
Let $\mathcal{U}(H)$ be the full unitary group of $B(H)$, and given an operator $X\in B(H)$ let $\mathcal{U}(X)$ denote the orbit of $X$ under $\mathcal{U}(H)$ acting by conjugation $X\mapsto UXU^*$.
With this notation we can state the classical Schur--Horn Theorem (\cite{Hor-1954-AJoM,Sch-1923-SBMG}) in a form which translates  naturally to the infinite dimensional case (see for instance Theorems~\ref{thm:kwpartialisometryorbit},~\ref{thm:kwrangeprojectionidentity} and Corollary~\ref{cor:conjecturetrueforinfinitemo}).
\begin{theorem}[Classical Schur--Horn Theorem \protect{\cite{Sch-1923-SBMG,Hor-1954-AJoM}}]
  \label{thm:schur-horn}
  Let $H$ be a finite dimensional complex Hilbert space and $\mathcal{D}$ a masa of $B(H)$ $(\cong M_n(\mathbb{C})$ relative to a fixed basis corresponding to $\mathcal{D})$ with conditional expectation $E:B(H)\to\mathcal{D}$.
Then for any self-adjoint operator $A\in B(H)$, 
  \[ E(\mathcal{U}(A)) = \{ B\in\mathcal{D}_{sa} \mid \lambda(B)\prec\lambda(A) \}, \]
where $\lambda(A),\lambda(B)$ denote the eigenvalue sequences of $A,B$ counting multiplicity.
\end{theorem}
Since the advent of the Schur--Horn theorem, there has been significant progress towards developing infinite dimensional analogues.
This was perhaps started by the work of Markus \cite{Mar-1964-UMN} and Gohberg and Markus \cite{GM-1964-MSN}, but more recently the topic was revived by A. Neumann in \cite{Neu-1999-JoFA}.
However, Neumann studied an approximate Schur--Horn phenomenon, i.e., the operator-norm closure of the diagonals of bounded self-adjoint operators  (equivalently, the $\ell^\infty$-norm closure of the diagonal sequences), which Arveson and Kadison deemed too coarse a closure \cite[Introduction paragraph 3]{AK-2006-OTOAaA}.
Instead, they studied the expectation of the trace-norm closure of the unitary orbit of a trace-class operator and then proved Schur--Horn analogues for trace-class operators in $B(H)$ (type I${}_\infty$ factor).
They also formulated a Schur--Horn conjecture for type II${}_1$ factors, but discussion of this topic is outside the scope of this paper.
For work on II${}_1$ and II${}_\infty$ factors, see the work of Argerami and Massey \cite{AM-2007-IUMJ,AM-2008-JMAA,AM-2013-PJM}, Bhat and Ravichandran \cite{BR-2014-PotAMS} and a recent unpublished work Ravichandran \cite{Rav-2012}.

Majorization plays an essential role in Schur--Horn phenomena, but also has given rise to new kinds of majorization inside and outside this arena. The Riemann and Lebesgue majorizations of Bownik--Jasper are applied to Schur--Horn phenomena in \cite{BJ-2013,BJ-TotAMS,Jas-2013-JoFA}. And the uniform Hardy--Littlewood majorization of Kalton--Sukochev \cite{KS-2008-JfudRuAM} is used ubiquitously in Lord--Sukochev--Zanin \cite{LSZ-2013} as an essential tool to study traces and commutators for symmetrically normed ideals. In this paper we also develop new kinds of majorization essential for our work (see introduction Figs. 1 and 2 and accompanying description). 

\paragraph{Basic notation for this paper.} 
For a set $S$,  let $\abs{S}$ denote its cardinality.
Let $c_0^+$ denote the cone of nonnegative sequences converging to zero and $c_0^*$ the cone of nonnegative decreasing sequences converging to zero.
For a sequence $\xi\in c_0^+$, let $\xi^*\in c_0^*$ denote the monotonization of $\xi$, or rather the monotonization of $\xi\vert_{\supp{\xi}}$ when $\xi$ is not finitely supported.
That is, $\xi^*_j$ denotes the $j$-th largest element of $\xi$.
Notice that if $\xi$ is finitely supported, then $\xi^*$ ends in zeros.
However, if $\xi$ has infinite support, then $\xi^*$ has \emph{no} zeros, and in this  case the monotonization $\xi^*$ reflects neither the zeros of $\xi$ nor their multiplicity. 

The following Definition~\ref{def:majorization} agrees with most of the literature, but it is a departure from that of \cite{KW-2010-JoFA} which did not include this equality condition.
When they needed an equality-like condition, they used instead the more restrictive Definition~\ref{def:strong-majorization} of strong majorization. 

\begin{definition}
  \label{def:majorization}
  Let $\xi,\eta\in c_0^+$.
One says that $\xi$ is \emph{majorized} by $\eta$, denoted $\xi\prec\eta$, if for all $n\in\mathbb{N}$,
\[ \sum_{j=1}^n \xi^*_j \le \sum_{j=1}^n \eta^*_j\quad\text{and}\quad \sum_{j=1}^\infty \xi_j = \sum_{j=1}^\infty \eta_j. \]
\end{definition}

\begin{definition}[\protect{\cite[Definition 1.2]{KW-2010-JoFA}}]
  \label{def:strong-majorization}
  Let $\xi,\eta\in c_0^+$.
One says that $\xi$ is \emph{strongly majorized} by $\eta$, denoted $\xi\preccurlyeq\eta$, if for all $n\in\mathbb{N}$,
\[ \sum_{j=1}^n \xi^*_j \le \sum_{j=1}^n \eta^*_j\quad\text{and}\quad \liminf_n\left\{\sum_{j=1}^n (\eta^*_j-\xi^*_j)\right\} = 0. \]
\end{definition}

Note that when $\xi\prec\eta\in\ell^1$, so is $\xi\in\ell^1$, and in this $\eta$-summable case, majorization \emph{as defined above} in Definition~\ref{def:majorization} is equivalent to strong majorization in Definition~\ref{def:strong-majorization}.
However, in the nonsummable case, the latter is clearly a stronger constraint than the former.
Strong majorization is not an essential tool in the main theorems of this paper, but we thought it important to emphasize the distinction between our definition of majorization and those of Kaftal--Weiss just described above.

The reason for our Definition~\ref{def:majorization} departure from that of Kaftal--Weiss is for convenience, efficiency of notation and unification of cases.
This notation allows us to state in a more unified way the results for both trace-class \emph{and} non trace-class operators simultaneously without splitting the conclusions into cases (compare Theorem~\ref{thm:kworthostochastic} to the two cases in \cite[Corollary 5.4]{KW-2010-JoFA}). 

\paragraph{Recent History.} 
In \cite{KW-2010-JoFA}, Kaftal and Weiss provided an exact extension of the Schur--Horn Theorem to positive compact operators,  i.e., precise characterizations without taking closures of any kind.
That is, in terms of majorization they characterize precisely the expectation of the unitary orbit of strictly positive compact operators and the expectation of the partial isometry orbit for all positive compact operators.
And they ask for but leave as an open question a characterization of the expectation of the unitary orbit of positive compact operators with nonzero kernel.

To describe this subject requires some traditional preliminaries.
The range projection $R_A$ for operators $A\in B(H)$ is the orthogonal projection onto $\overline{\range A}=\ker^\perp A^*$.
Thus for a self-adjoint operator $A$, $R_A^\perp$ is the projection onto $\ker A$ and hence $\trace R_A^\perp=\dim\ker A$, and in general $\trace R_A=\rank A$.
Throughout this paper we opt for using $\trace R_A$ and $\trace R_A^\perp$ instead of $\dim\ker A$ and $\rank A$. 
  \begin{definition}
    \label{def:partial-isometry-orbit}
    Given an operator $A\in B(H)$, the \emph{partial isometry orbit} of $A$ is the set
    \[ \mathcal{V}(A) = \{ VAV^* \mid V\in B(H), V^*V = R_A\vee R_{A^*} \}. \]
  \end{definition}

\noindent Notice this extends to partial isometries the standard notation of unitary orbits 
$\mathcal{U}(A) = \{ UAU^* \mid \text{unitary $U\in\mathcal{U}(H)$} \}$. 

Stochastic matrices play a central role in this subject due to the following definition and lemma. 

\begin{definition}\label{def:stochastics}
  A matrix $P$ with positive entries is called 
  \begin{itemize}[itemsep=0pt]
  \item {\emph{substochastic} if its row and column sums are bounded by 1;}
  \item {\emph{column-stochastic} if it is substochastic and its column sums equal 1;}
  \item {\emph{row-stochastic} if it is substochastic and its row sums equal 1;}
  \item {\emph{doubly stochastic} if it is row- and column-stochastic;}
  \item {\emph{unistochastic} if it is the Schur-product of a unitary matrix with its complex conjugate
      \hfill \\
      \hphantom{\quad}(the \emph{Schur-product} of two matrices $A = (a_{ij})$ and $B = (b_{ij})$ is the matrix $(a_{ij} b_{ij})$, that is, it is the entrywise product of $A,B$);
    }
    \item {\emph{orthostochastic} if it is the Schur-square of an orthogonal matrix, i.e., unitary with real entries.}
  \end{itemize}
\end{definition}
 
\noindent And the connection between expectations of orbits and stochastic matrices is:

\begin{lemma}[\protect{\cite[Lemmas 2.3, 2.4]{KW-2010-JoFA}}] \label{lem:kwstochasticcontraction}
  Let $\xi,\eta\in\ell^\infty$ and for any contraction $L=(L_{ij})\in B(H)$, let $Q_{ij} \coloneqq{} \left| L_{ij} \right|^2$ for all $i,j$.
Then 
\[\text{$\xi=Q \eta$ \quad if and only if \quad $\diag\xi = E(L\diag\eta L^*)$.} \]
Furthermore, 
  \begin{enumerate}[itemsep=0pt, label=\emph{(\roman*)}, labelwidth=1em, align=left] 
  \item {$Q$ is substochastic;}
  \item {$L$ is an isometry if and only if $Q$ is column-stochastic;}
  \item {$L$ is a co-isometry (isometry adjoint) if and only if $Q$ is row-stochastic;}
  \item {$L$ is unitary if and only if $Q$ is unistochastic;} 
  \item {$L$ is orthogonal if and only if $Q$ is orthostochastic.} 
  \end{enumerate}
\end{lemma}
\noindent
For completeness we repeat the straightforward short proof.
\begin{proof}
  Given $\xi,\eta\in\ell^\infty$, notice that for any $n\in\mathbb{N}$,
  \begin{align*}
    \angles{E(L\diag\eta\, L^*)e_n,e_n} & = \angles{L\diag\eta\, L^*e_n,e_n}                                                         \\
                                      & = \angles{(\diag\eta) \sum_{j=1}^\infty \bar{L}_{nj}e_j,\sum_{k=1}^\infty \bar{L}_{nk}e_k} \\ 
                                      & = \angles{\sum_{j=1}^\infty \eta_j\bar{L}_{nj}e_j,\sum_{k=1}^\infty \bar{L}_{nk}e_k}     \\
                                      & = \sum_{j=1}^\infty \abs{L_{nj}}^2 \eta_j = (Q\eta)_n. 
  \end{align*}
  Notice now that 
  \begin{equation}
    \label{eq:11}
    \sum_{j=1}^\infty Q_{ij} = \sum_{j=1}^\infty L_{ij}L^*_{ji} = \angles{LL^*e_i,e_i} = \snorm{L^*e_i}^2 \le 1 \qquad \text{for every $i$,} 
  \end{equation}
  and similarly 
  \begin{equation}
    \label{eq:12}
    \sum_{i=1}^\infty Q_{ij} = \norm{Le_j}^2 \le 1 \qquad \text{for every $j$.}
  \end{equation}
  \begin{enumerate}[label=(\roman*), align=left, labelwidth=1.2em, labelindent=0pt, labelsep=4pt, itemindent=4pt, leftmargin=1.2em, itemsep=0pt]
  \item Immediate from (\ref{eq:11}) and (\ref{eq:12}). \label{item:8}
  \item If $L$ is an isometry, then it is immediate from \ref{item:8} and the equality cases of (\ref{eq:12}) that $Q$ is column-stochastic.
Conversely assume that $Q$ is column-stochastic and hence $\snorm{Le_j}=1$ for all $j$ by (\ref{eq:12}).
Then $\angles{L^*Le_j,e_j}=1$ for all $j$ and thus it follows that $E(I-L^*L)=0$.
Since $E$ is faithful and $I-L^*L\ge 0$ because $L$ is a contraction by hypothesis, it follows that $L^*L=I$. \label{item:9}
  \item Apply \ref{item:9} to $L^*$. \label{item:5}
  \item Immediate from \ref{item:9} and \ref{item:5}. \label{item:4}
  \item Immediate from \ref{item:4} and the fact that $L$ has real entries. \qedhere
  \end{enumerate}
\end{proof}

Many of the results in \cite{KW-2010-JoFA} are stated and proved in terms of these stochastic matrices.
We state here some of their theorems more relevant to this study.
 
\begin{theorem}[\protect{\cite[Corollary 5.4]{KW-2010-JoFA}}] \label{thm:kworthostochastic}
  If $\xi,\eta\in c_0^*$, then
\[ \xi=Q\eta \text{ for some orthostochastic matrix }Q \iff \xi\prec\eta. \]
\end{theorem}
As an integration of the summable and nonsummable cases, above Theorem~\ref{thm:kworthostochastic} as stated is an example of the convenience afforded by our definition of majorization in contrast with that of \cite[Corollary 5.4]{KW-2010-JoFA}, where the summable and nonsummable cases are combined here under the new notation.

Using these tools, Kaftal and Weiss go on to prove an infinite dimensional analogue of the Schur--Horn Theorem for partial isometry orbits.
This includes the unitary orbits for strictly positive compact operators (see the next two theorems).
\begin{theorem}[\protect{\cite[Proposition 6.4]{KW-2010-JoFA}}] \label{thm:kwpartialisometryorbit}
  Let $A\in K(H)^+$.
Then
\[ E(\mathcal{V}(A)) =
  \{ B\in \mathcal{D}\cap K(H)^+ \mid s(B)\prec s(A) \}.
\]
\end{theorem}

Again, comparing this statement of the theorem with  \cite[Proposition 6.4]{KW-2010-JoFA}, one sees the convenience of defining majorization as in Definition~\ref{def:majorization}  as opposed to  \cite[Definition 1.2]{KW-2010-JoFA}. 

Focusing on the partial isometry orbit as opposed to the unitary orbit in the above theorem  sidesteps the effects of the dimension of the kernel of the operator $A$.
In this way, this theorem avoids the difficulties that lie therein.
A similar situation appeared in \cite{AK-2006-OTOAaA} when they studied $E(\overline{\mathcal{U}(A)}^{\norm{\cdot}_1})$ for positive trace-class operators $A$, which does not involve kernel dimension considerations.
 In addition, they showed that $\mathcal{V}(A)=\overline{\mathcal{U}(A)}^{\snorm{\cdot}_1}$. 
More generally, for $A$ positive compact, it is elementary to show $\mathcal{V}(A)=\overline{\mathcal{U}(A)}^{\Vert\cdot\Vert}$.
Since none of these three objects encodes the dimension of the kernel of $A$ and because they coincide for trace-class operators, $\mathcal{V}(A)$ is a natural substitute for $\overline{\mathcal{U}(A)}^{\snorm{\cdot}_1}$ for positive compact operators $A$ (outside the trace class) in the context of Schur--Horn theorems. 

However, the question of precisely what is $E(\mathcal{U}(A))$ for all $A\in K(H)^+$ was only partially answered in \cite{KW-2010-JoFA}.
In particular, it was answered when $A$ has finite rank or when $R_A=I$ (that is, when $A$ is strictly positive).
When $A$ has finite rank, $\mathcal{U}(A)=\mathcal{V}(A)$ (see Theorem \ref{thm:sufficiencyofpmajorization}, proof case 1), and so is covered by Theorem~\ref{thm:kwpartialisometryorbit}.
For the case when $R_A=I$,  they have
\begin{theorem}[\protect{\cite[Proposition 6.6]{KW-2010-JoFA}}] \label{thm:kwrangeprojectionidentity}
  Let $A\in K(H)^+$ with $R_A=I$.
Then 
\[ E(\mathcal{U}(A)) = E(\mathcal{V}(A))\cap\{ B\in\mathcal{D} \mid R_B=I \}. \]
\end{theorem}

\paragraph{Our main contribution.}  
Theorem~\ref{thm:kwrangeprojectionidentity} left open the case when $A$ has infinite rank and nonzero kernel.
We attempt here to close this gap.
In particular, we characterize $E(\mathcal{U}(A))$ when $A$ has infinite dimensional kernel.
When $A$ has finite dimensional kernel, we give a necessary condition for  membership in $E(\mathcal{U}(A))$ (which we conjecture is also sufficient), and we give a sufficient condition for  membership in $E(\mathcal{U}(A))$ (which we know not to be necessary when $0<\trace R_A^\perp <\infty$, see Example~\ref{ex:nonnecessity-of-p-majorization} below which also appears in \cite[Proposition 6.10, Example 6.11]{KW-2010-JoFA}).
These main results are embodied in Theorems~\ref{thm:sufficiencyofpmajorization}, \ref{thm:necessityofapppmajorization} and Corollary~\ref{cor:conjecturetrueforinfinitemo}.
Both of these membership conditions involve new kinds of majorization, which here we call $p$-majorization and herein we introduce approximate $p$-majorization (for $0\le p\le \infty$, Definitions~\ref{def:pmajorization} and~\ref{def:approximatepmajorization} below).
There is a natural hierarchy of these new types of majorization which the diagram in Figure~\ref{fig:maj-hierarchy} describes succinctly.
All of these implications are natural (see the discussions following Definitions~\ref{def:pmajorization} and \ref{def:approximatepmajorization}) except the two corresponding to the dashed arrows, which are handled in Proposition~\ref{prop:strong-maj-vs-p-maj} and are only applicable when both sequences in question are in $c_0^+\setminus\ell^1$. A linear interpretation of the diagram in Figure~\ref{fig:maj-hierarchy} is presented in Figure~\ref{fig:lin-maj-hierarchy}. 

\begin{figure}[h!]
  \centering
  \includegraphics{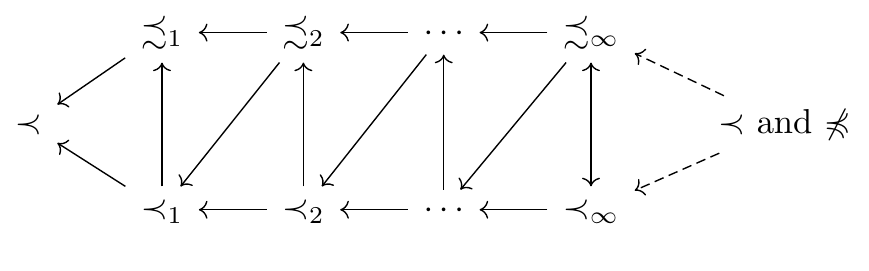}
  \caption{Hierarchy of majorization}
  \label{fig:maj-hierarchy}
\end{figure}

\begin{figure}[h!]
  \centering
  \includegraphics{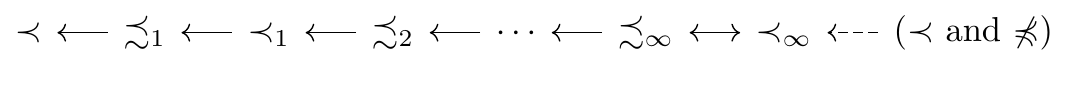}
  \caption{Linear hierarchy of majorization}
  \label{fig:lin-maj-hierarchy}
\end{figure}

\section{$p$-majorization (sufficiency)} \label{sec:p-majorization}

A result which was known to Kaftal and Weiss, and almost certainly to others, is a necessary condition for membership in the expectation of the unitary orbit of a positive operator (see \cite[Proof of Lemma 6.9]{KW-2010-JoFA}).
Namely, the dimensions of the kernels of operators in the range of the expectation of the unitary orbit of a positive operator cannot increase from the dimension of the kernel of the operator itself.

\begin{proposition}\label{prop:cantaddzeros}
If $A\in B(H)^+$ and $B\in E(\mathcal{U}(A))$, then $\ker B \subseteq \ker(UAU^*)$ for some $U\in \mathcal{U}(H)$, and hence also $\trace R_B^\perp \le\trace R_A^\perp $. 
\end{proposition}

\begin{proof}[Vector proof]
  Since $B\in E(\mathcal{U}(A))$, $B=E(UAU^*)$ for some unitary $U\in\mathcal{U}(H)$, and since $\trace R_A^\perp =\dim\ker A =\dim\ker UAU^* = \trace R_{UAU^*}^\perp $, the required trace inequality follows from the inclusion $\ker B\subseteq \ker UAU^*$, which itself follows from $A\ge 0$ and 
  \begin{align*}
    \ker B &= \ker E(UAU^*) \\
    &= \overline{\spans}\{e_n\mid \angles{E(UAU^*)e_n,e_n}=0\} \\
    &= \overline{\spans}\{e_n\mid \angles{UAU^*e_n,e_n}=0\} \\
    &= \overline{\spans}\{e_n\mid \snorm{(UAU^*)^{\nicefrac{1}{2}}e_n}^2=0\} \\
    &\subseteq \ker(UAU^*)^{\nicefrac{1}{2}} = \ker UAU^* \qedhere
  \end{align*}
\end{proof}

\begin{proof}[Projection proof]
  Since $A\ge 0$, $R_A^\perp$ is the largest projection $P$ so that $PAP = 0$.
Since $B\in E(\mathcal{U}(A))$, there is some unitary $U\in\mathcal{U}(H)$ so that $B=E(UAU^*)$.
Notice also that $R_B^\perp \in \mathcal{D}$ since $B\in\mathcal{D}$.
Therefore, 
  \begin{align*}
    E(U(U^*R_B^\perp U)A(U^*R_B^\perp U)U^*) 
    &= E(R_B^\perp UAU^*R_B^\perp) \\
    &= R_B^\perp E(UAU^*)R_B^\perp = R_B^\perp BR_B^\perp=0.
  \end{align*}
Then since $E$ is faithful and $U(U^*R_B^\perp U)A(U^*R_B^\perp U)U^*$ is positive, it is zero, and thus conjugating by $U^*$ shows $(U^*R_B^\perp U)A(U^*R_B^\perp U)=0$.
By the maximality of $R_A^\perp$, $U^*R_B^\perp U\le R_A^\perp$ and therefore
\[ \trace R_B^\perp  = \trace(U^*R_B^\perp U) \le \trace R_A^\perp . \qedhere \]
\end{proof}

Before we proceed with our analysis, we need to introduce next a concept similar to \cite[Definition 6.8(ii)]{KW-2010-JoFA} which here we call $p$-majorization.
 Roughly speaking, it is majorization along with eventual $p$-expanded majorization.
And this led us to the definition below of $\infty$-majorization which is both new and fruitful. 

\begin{definition}\label{def:pmajorization}
Given $\xi,\eta\in c_0^+$ and $0\le p<\infty$, we say that $\xi$ is  \emph{$p$-majorized} by $\eta$, denoted $\xi\prec_p\eta$, if $\xi\prec\eta$ and there exists an $N_p\in\mathbb{N}$ such that for all $n\ge N_p$, we have the inequality
\[ \sum_{k=1}^{n+p} \xi_k \le \sum_{k=1}^{n} \eta_k. \]
And $\infty$-majorization, denoted $\xi\prec_\infty\eta$, means $\xi\prec_p\eta$ for all $p\in\mathbb{N}$. 
\end{definition}

\noindent Note that $\xi\prec_0\eta$ is precisely the statement that $\xi\prec\eta$ (recall Definition~\ref{def:majorization}, which includes equality of the sums).
One  also observes that if $\xi\prec_p\eta$ and $p'\le p$, then $\xi\prec_{p'}\eta$ (we use often the special case that $p$-majorization implies $0$-majorization, i.e., majorization).
For this reason, $\xi\prec_p\eta$ for infinitely many $p$ is equivalent to $\xi\prec_p\eta$ for all $p<\infty$, in which case we say that $\xi$ is $\infty$-majorized by $\eta$ and we write $\xi\prec_\infty\eta$. 

One should also take note that $p$-majorization is actually strictly stronger than $p'$-majorization when $p'<p$. 
That is, there exist sequences $\xi,\eta\in c_0^+$ for which $\xi\prec_{p'}\eta$ but $\xi\not\prec_p\eta$.
 From the remarks of the previous paragraph, it suffices to exhibit $\xi,\eta$ when $p=p'+1$. 
To produce such sequences, start with any $0<\eta\in c_0^*$ and define 
\[ \xi^{(p)} \coloneqq{} \langle \underbrace{\nicefrac{\eta_1}{p},\ldots,\nicefrac{\eta_1}{p}}_{\text{$p$ times}},\eta_2,\eta_3,\ldots \rangle. \]
Even though $\xi^{(p)}$ is not necessarily monotone, it is not difficult to verify that $\xi^{(p)}\prec_{p'}\eta$ but $\xi^{(p)}\not\prec_p\eta$. 

\begin{remark}
  \label{rem:strategy}
  Because $\mathcal{V}(A)\supseteq \mathcal{U}(A)$, Theorem~\ref{thm:kwpartialisometryorbit} for $A\in K(H)^+$ implies that $s(B)\prec s(A)$ is a necessary condition for $B\in E(\mathcal{U}(A))$;  Proposition~\ref{prop:cantaddzeros} shows $\trace R_B^\perp\le \trace R_A^\perp$ is another necessary condition for membership in $E(\mathcal{U}(A))$ and with majorization is equivalent to membership in $E(\mathcal{U}(A))$ when $\trace R_A^\perp=0$ by Theorem~\ref{thm:kwrangeprojectionidentity}.
It was natural in \cite{KW-2010-JoFA} to ask how the role of majorization is impacted by the dimension of these kernels.
In particular here we enhance this program by asking, what role, if any, $p=\trace R_A^\perp-\trace R_B^\perp$ plays in relating membership in $E(\mathcal{U}(A))$ to majorization.
And when this difference is undefined, $p=0$ is the minimal $p$ for which $\trace R_A^\perp\le \trace R_B^\perp +p$.
This is the strategy that guided our program.
\end{remark}

The result below appears in \cite{KW-2010-JoFA} as Lemma 6.9 for $p<\infty$, the proof of which utilizes orthostochastic matrices.
We provide a different proof which instead utilizes expectations of unitary orbits because it leads to a straightforward extension to the $p=\infty$ case.
See Remark~\ref{rem:pmajorizationimpliesorthostochastic} for when an orthostochastic matrix can be produced to implement the construction.
But we do not have a complete characterization for this orthostochasticity case.

\begin{theorem}\label{thm:sufficiencyofpmajorization}
Let $A,B\in K(H)^+$, $B\in\mathcal{D}$ and  $\trace R_B^\perp \le\trace R_A^\perp $. \\ If for some $0\le p\le\infty$, $s(B)\prec_p s(A)$ and $\trace R_A^\perp \le\trace R_B^\perp +p$, \\ then $B\in E(\mathcal{U}(A))$. 
\end{theorem}

\begin{proof}

Without loss of generality we may assume that $A\in\mathcal{D}$ (i.e., $A,B$ are simultaneously diagonalized) because $\mathcal{U}(A)=\mathcal{U}(UAU^*)$ for every $U\in\mathcal{U}(H)$. 
We may further reduce to the case when $\trace R_B^{\perp}=0$ via a splitting argument. 
Because $\dim\ker B=\trace R_B^\perp<\trace R_A^\perp=\dim\ker A$, without loss of generality we can assume $\ker B\subseteq\ker A$. 
Then with respect to $H=S\oplus S^\perp$, one has $A=A_1\oplus A_2$, $B=B_1\oplus B_2$, $A_1=0=B_1$, and $B_2$ has zero kernel. 
Thus once we find a unitary $U$ on $S^{\perp}$ for which $E(UA_2U^{*})=B_2$, then $I\oplus U$ satisfies $E((I\oplus U)A(I\oplus U)^{*})=B$.

\emph{Case 1: $A$ has finite rank.}  

\noindent In this case $\mathcal{U}(A)=\mathcal{V}(A)$, even if $A$ is not necessarily self-adjoint.
Indeed, the elements of $\mathcal{V}(A)$, by Definition \ref{def:partial-isometry-orbit}, have the form $VAV^*$ for some partial isometry for which $V^*V=R_A \vee R_{A^*}(\coloneqq{}P)$, so $PA=AP=A$.
Then $VAV^*=(VP)A(VP)^*$ and $VP$ is also a partial isometry with $(VP)^*(VP)=P$.
But this partial isometry $VP$ is finite rank since $R_A$ and $R_{A^*}$ and hence also $P$ are finite rank, and so $VP$ can be extended to a unitary $U$ for which $VAV^*=(VP)A(VP)^*=UAU^*$.
This shows $\mathcal{V}(A)\subseteq\mathcal{U}(A)$, and hence equality.
That the conclusion of Theorem \ref{thm:sufficiencyofpmajorization} holds in this case is then covered by Theorem \ref{thm:kwpartialisometryorbit}, since $A\ge 0$, $\mathcal{U}(A)=\mathcal{V}(A)$, and since $s(B)\prec_p s(A)$ implies $s(B)\prec s(A)$. 

In the proof of Case 1, we only used the facts that $A$ had finite rank and $s(B)\prec s(A)$.
Although not needed in this case, the other hypotheses hold automatically and for edification we explain why.
Even though $s(B)\prec_p s(A)$ for some $p\ge 0$ implies $s(B)\prec s(A)$, one has the stronger converse: when $s(A)$ has finite support, $s(B)\prec s(A)$ implies $s(B)\prec_p s(A)$ for every $p\ge 0$, i.e., $s(B)\prec_\infty s(A)$.
Indeed, let $N_p$ be the largest index for which $s(A)$ has a nonzero value.
Then for all $k\ge N_p$, we have 
\[ \sum_{j=1}^{k+p} s_j(B) \le \sum_{j=1}^{k+p} s_j(A) = \sum_{j=1}^k s_j(A) \]
and therefore $s(B)\prec_p s(A)$.
Since $p$ is arbitrary, $s(B)\prec_\infty s(A)$.
This shows that the second inequality in the hypotheses is satisfied for $p=\infty$ since its right-hand side is infinite.
The first inequality is satisfied since $\trace R_A^\perp =\infty$ because $A$ is finite rank. 

\emph{Case 2: $A$ has infinite rank and $\trace R_B^\perp =\trace R_A^\perp $ (the $p=0$ case).} 

\noindent With the earlier reduction that $\trace R_B^{\perp}=0$ and using $p=0$, Case 2 is a direct consequence of Theorem~\ref{thm:kwrangeprojectionidentity}.

\emph{Case 3: $A$ has infinite rank and $\trace R_B^\perp <\trace R_A^\perp $ (the most complicated case).}

\noindent Since $\trace R_A^\perp\le \trace R_B^\perp +p$, if necessary, by passing to a smaller $p$ we may assume $\trace R_A^\perp =\trace R_B^\perp +p$, even if $\trace R_A^\perp = \infty$ (in which case $p=\infty$). 
Since $\trace R_B^{\perp}=0$, one has $\trace R_A^\perp = p$, and since $\trace R_B^\perp<\trace R_A^\perp$, one has $1\le p\le \infty$.

We now employ another splitting. 
First let $\{e_j\}_{j=1}^\infty$ denote the basis that diagonalizes $A,B$, and then assign them different names, that is, let $\{f_j\}_{j=1}^p$ be the collection of $e_j$'s such that $(Ae_j,e_j)=0$.
Since $A$ is diagonalized with respect to the basis $\{e_j\}_{j=1}^\infty$ the collection $\{f_j\}_{j=1}^p$ forms an orthonormal basis for $\ker A$.
Let $\{g_j\}_{j=1}^\infty$ consist of the remainder of the set $\{e_j\}_{j=1}^\infty$, which means $\{g_j\}_{j=1}^\infty$ is a basis for $\ker^\perp A$.
Then $H=H_1\oplus H_2$, where $H_1=\ker A$ and $H_2=\ker^\perp A$.
Let $\mathbb{N}_p\coloneqq{}\{n\in\mathbb{N} \mid n\le p\}$ and then define $\diag_{H_1\oplus H_2}:\ell^\infty(\mathbb{N}_p)\times\ell^\infty(\mathbb{N})\to \mathcal{D}$ by
\begin{equation}
  \label{eq:21}
  \diag_{H_1\oplus H_2}\angles{\phi,\rho} = \left(
    \begin{array}{ccc|ccc}
      \phi_1 & \cdots & 0      & 0      & \cdots & 0      \\ 
      \vdots & \ddots & \vdots & \vdots & \ddots & \vdots \\
      0      & \cdots & \phi_p & 0      & \cdots & 0      \\ \hline
      0      & \cdots & 0      & \rho_1 & \cdots & 0      \\
      \vdots & \ddots & \vdots & \vdots & \rho_2 & \vdots \\
      0      & \cdots & 0      & 0      & \cdots & \ddots \\
    \end{array}
    \right)
\end{equation}
Then because of the way in which we chose $\{f_j\}_{j=1}^p$ and $\{g_j\}_{j=1}^\infty$ and because $\trace R_B^{\perp}=0$, if we let $\eta'\coloneqq{}s(A)$ and $\xi=s(B)$ then $A,B$ are 
\[ A = \diag_{H_1\oplus H_2}\angles{\mathbf{0},\eta'} \qquad \text{and} \qquad B = \diag\xi, \]
where $\mathbf{0}\in\ell^\infty(\mathbb{N}_p)$ is the zero sequence and by hypothesis $\xi\prec_p\eta'$.

The heuristic idea of the proof is the following in descriptive informal language. 
First construct a sequence $\xi'$ which is a sparsely compressed version of $\xi$ but sufficient to retain majorization by $\eta$, i.e., $\xi'\prec \eta$.
Next apply Case 2 to obtain a special unitary $U$ for which $E(U(\diag\angles{\mathbf{0},\eta})U^*)=\diag\angles{\mathbf{0},\xi'}$.
Finally, apply another unitary to decompress the diagonal $\angles{\mathbf{0},\xi'}$ to the diagonal $\xi$.

Now inductively choose sequences of nonnegative integers $\{N_m\}_{m=1}^p$ and $\{N'_m\}_{m=0}^p$ with the following properties:
\begin{enumerate}[itemsep=0pt,label=(\roman*)]
\item $0=N'_0<N'_0+1<N_1<N'_1<N'_1+1<N_2<N'_2<\cdots$ \quad if $p=\infty$,  \\
  \hphantom{$0=N'_0<N'_0+1<N_1<N'_1<N'_1+1+$}(or $\cdots<N_p<N'_p$ if $p<\infty$). \label{item:1}
\item For each $m\in\mathbb{N}_p$, whenever $n\ge N_m-(m-1)$, one has
\[ \sum_{k=1}^{n+m} \xi_k \le \sum_{k=1}^n \eta'_k. \] \label{item:2}
\item $\xi_{N_m}+\xi_{N'_m} \le \xi_{N_m-1}$ for each $m\in\mathbb{N}_p$. \label{item:3}
\end{enumerate}
For transparency and brevity, we only loosely describe the construction of the pair of sequences $\{N_m\}_{m=1}^p$ and $\{N'_m\}_{m=0}^p$. 
The construction proceeds in pairs, $N_m,N'_m$. 
We may choose $N_m$ to satisfy property~\ref{item:2} since $\xi\prec_p\eta'$ and hence $\xi\prec_m\eta'$ because $m\le p$.
For this we use the fact that property~\ref{item:2} is an eventual property in the sense that if it holds for some $N_m$ it holds for any larger $N_m$. 
Moreover, because $\xi>0$ and $\xi\downarrow 0$, $\xi$ has infinitely many strictly decreasing jumps, i.e., for infinitely many $j$ one has $\xi_j \lneqq \xi_{j-1}$. 
If necessary, increase $N_m$ so that it satisfies this condition. 
Then since $\xi\to 0$, we may choose $N'_m$ to satisfy property~\ref{item:3}. 
To construct the entire pair of sequences, simply iterate this process while simultaneously ensuring $N_{m+1}>N'_m+1$, which we can guarantee because property~\ref{item:2} is an eventual property.

Next since $N'_0<N_m<N'_m<N_{m+1}<N'_{m+1}$ for all $1\le m<p$, one has for $p=\infty$
 \begin{align*}
   N'_0+1 = 1 &\le N_1<N'_1\le N_2-1<N'_2-1\le N_3-2 < N'_3-2\le  \cdots \\ 
   \cdots &\le N_m-(m-1)<N'_m-(m-1)\le
   N_{m+1}-m<N'_{m+1}-m\le \cdots \\ 
   \Big(\cdots &\le N_p-(p-1)<N'_p-(p-1)\quad\text{for $p<\infty$}\Big).
 \end{align*}
 When $p<\infty$, if we set $N'_{p+1}=\infty$ for convenience of notation, then regardless of whether $p<\infty$ or $p=\infty$ these inequalities partition
\[ \mathbb{N} = \bigsqcup_{m=0}^p \left[N'_m-(m-1),N'_{m+1}-m\right) \]
with each $N_m-(m-1)\in \left[N'_{m-1}-(m-2),N'_m-(m-1)\right)$ and $m\in\mathbb{N}_p$. 

Next define the sequence $\xi'$ which shifts and alters $\xi$ at one point in each $\left[N'_{m-1}-(m-2), N'_m-(m-1)\right)$: for each $m\in\mathbb{N}_p$ (or for each $m\in\mathbb{N}_{p+1}$ if $p<\infty$, in which case the last interval is $[N'_p-(p-1),\infty)$), set

\[ \xi'_k = 
\begin{cases} 
  \xi_{N_m}+\xi_{N'_m} & \text{if $k=N_m-(m-1)$}; \\ 
  \xi_{k+m-1} & \parbox[t]{0.5\textwidth}{if $N'_{(m-1)}-(m-2)\le k<N'_m-(m-1)$ but $k\not=N_m-(m-1)$.} \\ 
\end{cases} 
\] 

This partition of $\mathbb{N}$ ensures that $\xi'$ is well-defined.
Property~\ref{item:3} guarantees that $\xi'$ is monotone decreasing.
And property~\ref{item:2} allows us to conclude that $\xi'\prec \eta'$ which will follow from equations \eqref{eq:17}--\eqref{eq:23}.
We omit their straightforward natural proofs for the sake of clarity of exposition. Equations \eqref{eq:17}-\eqref{eq:22} have natural proofs by induction and \eqref{eq:23} is merely one case for when $p<\infty$.

  For all $m\in\mathbb{N}_p$, 
  \begin{equation}
    \label{eq:17}
    \sum_{j=1}^k \xi'_j = \sum_{j=1}^{k+m-1} \xi_j \quad\text{for}\quad N'_{m-1}-(m-2)\le k<N_m-(m-1),
  \end{equation}
  and 
  \begin{equation}
    \label{eq:22}
    \sum_{j=1}^k \xi'_j  = \xi_{N'_m} + \sum_{j=1}^{k+m-1} \xi_j \quad\text{for}\quad N_m-(m-1)\le k<N'_m-(m-1).
  \end{equation}
   
  When $p<\infty$, the last interval requires separate consideration. That is, if $N'_p-(p-1)\le k<N'_{p+1}-p=\infty$ and $N'_p-(p-1)\le j\le k$ one has $\xi'_j=\xi_{j+p}$ and so
  \begin{equation}
    \label{eq:23}
    \sum_{j=1}^k \xi'_j = \sum_{j=1}^{k+p} \xi_j \quad \text{for} \quad k\in [N'_p-(p-1),\infty)
  \end{equation}

It is now simple to prove that $\xi'\prec\eta'$ using condition~\ref{item:2} and equations \eqref{eq:17}--\eqref{eq:23}.
Indeed, notice that for $m\in\mathbb{N}_p$, if $N'_{m-1}-(m-2)\le k < N_m-(m-1)$ and since also $m-1\in\mathbb{N}_p\cup\{0\}$ and $\xi\prec\eta'$, and considering separately the cases $m=1$ and $m>1$, one has
\begin{equation}
  \label{eq:18}
  \sum_{j=1}^k \xi'_j \ \underset{\text{(\ref{eq:17})}}{=} \ \sum_{j=1}^{k+m-1} \xi_j \  \underset{\underset{\text{or \ref{item:2}}}{\xi\prec\eta'}}{\le} \ \sum_{j=1}^k \eta'_j.
\end{equation}
And for $m\in\mathbb{N}_p$, if $N_m-(m-1) \le k < N'_m -(m-1)$, one has 
\begin{equation}
  \label{eq:16}
  \sum_{j=1}^k \xi'_j \  \underset{\text{(\ref{eq:22})}}{=} \ \xi_{N'_m} + \sum_{j=1}^{k+m-1} \xi_j \ \le \ \sum_{j=1}^{k+m} \xi_j \ \underset{\text{~\ref{item:2}}}{\le} \ \sum_{j=1}^k \eta'_j.
\end{equation}
Finally, if $p<\infty$ and $N'_p-(p-1)\le k<N'_{p+1}-p = \infty$, one has
\begin{equation}
  \label{eq:15}
  \sum_{j=1}^k \xi'_j \underset{\text{(\ref{eq:23})}}{=} \sum_{j=1}^{k+p} \xi_j \underset{\text{~\ref{item:2}}}{\le} \sum_{j=1}^k \eta'_j.
\end{equation}
Recalling that $\mathbb{N}$ is the disjoint union of $\big[N'_{(m-1)}-(m-2),N'_m-(m-1)\big)$ for $m\in\mathbb{N}_p$,  (\ref{eq:17}--\ref{eq:23}) and (\ref{eq:18}--\ref{eq:15}) imply that for all $k\in\mathbb{N}$
\[ \sum_{j=1}^k \xi_j \underset{\text{~(\ref{eq:17}--\ref{eq:23})}}{\le} \sum_{j=1}^k \xi'_j \underset{\text{~(\ref{eq:18}--\ref{eq:15})}}{\le} \sum_{j=1}^k \eta'_j. \]
Passing to the limit as $k\to\infty$ yields $\sum_{j=1}^\infty \xi_j = \sum_{j=1}^\infty \xi'_j = \sum_{j=1}^\infty \eta'_j$ since $\xi\prec\eta'$.
Hence $\xi'\prec\eta'$. 
\vspace{0.2cm}

By Case 2 and our earlier reduction (first paragraph of proof) applied to $A=\diag_{H_1\oplus H_2}\angles{\mathbf{0},\eta'}$ and $B'=\diag_{H_1\oplus H_2}\angles{\mathbf{0},\xi'}$, we obtain a unitary $U$ of the form $I_{H_1}\oplus W$ for which
\[ E(UAU^*) = B' = \diag_{H_1\oplus H_2}\angles{\mathbf{0},\xi'}. \]
Since $U$ has this form, for all $f,f'\in\{f_j\}_{j=1}^p$ and $g\in\{g_j\}_{j=1}^\infty$ one has
\[ 0 = \angles{UAU^* f,f'} = \angles{UAU^* f,g} = \angles{UAU^* g,f}. \]
Then for all $m\in\mathbb{N}_p$ define $V_m$ on $\spans\{f_m,g_{N_m-(m-1)}\}$ given by the unitary $2\times 2$ matrix
\begin{equation}
  \label{eq:20}
  V_m =  
  \begin{pmatrix} 
    a_m  & b_m \\
    -b_m & a_m \\
  \end{pmatrix}
\end{equation}
where 
\[ a_m = \sqrt{\frac{\xi_{N'_m}}{\xi_{N_m}+\xi_{N'_m}}} \qquad\text{and}\qquad b_m = \sqrt{\frac{\xi_{N_m}}{\xi_{N_m}+\xi_{N'_m}}}. \]
Then $C_m =
\begin{psmallmatrix}
  0 & 0 \\
  0 & \xi'_{N_m-(m-1)} \\
\end{psmallmatrix}
$ is the compression of $UAU^*$ to $\spans\{f_m,g_{N_m-(m-1)}\}$ and if one interprets $V_m$ as canonically acting on this same subspace, one computes
\begin{equation}
  \label{eq:13}
  V_mC_mV_m^* =   
  \left(
    \begin{matrix}
      \xi_{N_m} & {\Huge *} \\ {\Huge *} & \xi_{N'_m} \\
    \end{matrix}
  \right)
\end{equation}
because
\begin{align}
  V_m C_m V_m^*  
  &=  \nonumber
  \begin{pmatrix} 
    a_m  & b_m \\
    -b_m & a_m \\
  \end{pmatrix}
  \begin{pmatrix} 
    0 & 0                \\
    0 & \xi'_{N_m-(m-1)} \\
  \end{pmatrix}
  \begin{pmatrix} 
    a_m & -b_m \\
    b_m & a_m  \\
  \end{pmatrix}
  \\
  &= \nonumber
  \begin{pmatrix} 
    0 & b_m(\xi_{N_m}+\xi_{N'_m}) \\
    0 & a_m(\xi_{N_m}+\xi_{N'_m}) \\
  \end{pmatrix}
  \begin{pmatrix} 
    a_m & -b_m \\
    b_m & a_m  \\
  \end{pmatrix}
  \\
  &= \nonumber
  \begin{pmatrix} 
    b_m^2(\xi_{N_m}+\xi_{N'_m}) & a_mb_m(\xi_{N_m}+\xi_{N'_m}) \\
    a_mb_m(\xi_{N_m}+\xi_{N'_m}) & a_m^2(\xi_{N_m}+\xi_{N'_m}) \\
  \end{pmatrix}
  \\
  &= \nonumber
  \left(
    \begin{matrix}
      \xi_{N_m} & {\Huge *} \\ {\Huge *} & \xi_{N'_m} \\
    \end{matrix}
  \right).
\end{align}

Next let $H'=H\circleddash\spans\{f_m,g_{N_m-(m-1)}\}_{m=1}^p$ and let $V$ be the unitary defined by
\[ V = I_{H'} \oplus \bigoplus_{m=1}^p V_m. \]
From the above computations, one can see that the diagonal operator
\[ E(VUAU^*V^*) = \Pi^* (\diag\xi) \Pi, \]
for an appropriate permutation $\Pi$ of the basis $\{e_n\}_{n\in\mathbb{N}}$.
But conjugation by operators which permute the basis corresponding to $E$, in particular $\Pi$, commutes with $E$, and $\Pi VU$ is unitary, so  $E((\Pi VU)A(\Pi VU)^*)=\diag\xi$. 
\end{proof}

\begin{remark}[\protect{\bfseries Orthostochasticity}] \label{rem:pmajorizationimpliesorthostochastic}
  In the above proof, if $A$ was already diagonalized with respect to the basis $\{e_n\}_{n\in\mathbb{N}}$, then all the unitary operators either are orthogonal with respect to this basis, or can be chosen as such.
Indeed, $\Pi$ and $V$ are orthogonal, and $U=I_{H_1}\oplus W$ where $W$, coming from Theorem~\ref{thm:kwrangeprojectionidentity}, can be chosen to be orthogonal by \cite[Corollary 6.1, NS(ii$'$) and S(ii$'$)]{KW-2010-JoFA}.
A consequence of this combined with Lemma~\ref{lem:kwstochasticcontraction} is that if $\xi,\eta\in c_0^+$, $\xi\prec_p\eta$ and if
\[ \abs{\xi^{-1}(0)} \le \abs{\eta^{-1}(0)} \le p + \abs{\xi^{-1}(0)}, \]
then there exists an orthostochastic matrix $Q$ for which  $\xi=Q\eta$. 
\end{remark}

\begin{example}[\protect{\cite[Proposition 6.10, Example 6.11]{KW-2010-JoFA}}]
 \label{ex:nonnecessity-of-p-majorization}
The following example, in conjunction with Lemma 1.6, will show that the condition $s(B)\prec_p s(A)$ is not necessary for $B\in E(\mathcal{U}(A))$ and  $p\coloneqq{}\trace R_A^\perp -\trace R_B^\perp <\infty$.
In particular, the converse of Theorem~\ref{thm:sufficiencyofpmajorization} fails for this case, but is true when $\trace R_A^\perp -\trace R_B^\perp$ is either infinite or undefined as proved by Corollary~\ref{cor:conjecturetrueforinfinitemo}.

A counterexample is the following. Let $\tilde\eta = \angles{0,\eta}$ where $0<\eta\in c_0^*$.
Let $Q$ be an orthostochastic matrix with the property that $Q_{ij}=0$ if and only if $i>j>1$ (so that otherwise $Q_{ij}>0$ and rows and columns sum to one), e.g., \cite[Example 6.11]{KW-2010-JoFA}. 
Then choosing $\xi\coloneqq{}Q\tilde{\eta}$ we claim that $\xi \not\prec_1 \eta$.
One can see this from the calculation
  \begin{align}
    \sum_{i=1}^n \xi^*_i 
    &\ge \sum_{i=1}^n \xi_i = \sum_{i=1}^n \sum_{j=1}^\infty Q_{ij}\tilde\eta_j \nonumber \\ 
    &= \sum_{i=1}^n \sum_{j=2}^\infty Q_{ij}\eta_{j-1} \nonumber \\
    &= \sum_{j=2}^n \sum_{i=1}^n Q_{ij}\eta_{j-1} + \sum_{j=n+1}^\infty \sum_{i=1}^n Q_{ij}\eta_{j-1} \nonumber \\
    &= \sum_{j=1}^{n-1} \eta_j + \sum_{j=n+1}^\infty \sum_{i=1}^n Q_{ij}\eta_{j-1} > \sum_{j=1}^{n-1} \eta_j, \label{eq:35}
  \end{align}
where the latter inequality follows since $\eta>0$. 

This example can be easily extended to create similar orthostochastic examples (e.g., 
$\tilde Q = \bigoplus_1^{p-1}\begin{psmallmatrix} 
\nicefrac{1}{2} & \nicefrac{1}{2} \\ 
\nicefrac{1}{2} & \nicefrac{1}{2} \\ 
\end{psmallmatrix}\oplus Q$) 
for when $1<p<\infty$ and $0<\xi \not\prec_p \eta$, but $\xi = \tilde Q\tilde\eta$ for $\tilde\eta = \angles{0,\eta_1,0,\eta_2,\ldots,0,\eta_p,\eta_{p+1},\ldots}$ with $p$ zeros.
Then since $\tilde Q$ is orthostochastic, by Lemma~\ref{lem:kwstochasticcontraction} one has that $\diag\xi=U(\diag\tilde\eta)U^*$ for some orthogonal matrix $U$ with $\tilde Q_{ij}=\abs{U_{ij}}^2$.
Therefore $p$-majorization for any $1\le p<\infty$ is not necessary to characterize $E(\mathcal{U}(A))$. 

To verify $\xi\not\prec_p\eta$ observe that for sufficiently large $n$, one has
\begin{align*}
  \sum_{i=1}^n \xi^*_i &= \sum_{i=1}^{p-1} \eta_i + \sum_{i=1}^{n-2p+2} (Q\angles{0,\eta_p,\eta_{p+1},\ldots}^*)_i \\
  &\underset{(\ref{eq:35})}{>} \sum_{i=1}^{p-1} \eta_i + \sum_{i=1}^{n-2p+1} \eta_{i+p-1} \\
  &= \sum_{i=1}^{p-1} \eta_i + \sum_{i=p}^{n-p} \eta_i =  \sum_{i=1}^{n-p} \eta_i. \\
\end{align*}

\end{example}

\paragraph{Discussion on majorizations.} 
The Figures~\ref{fig:maj-hierarchy}--\ref{fig:lin-maj-hierarchy} at the end of the introduction show the interconnections between various types of majorization.
Here through Proposition~\ref{prop:strong-maj-vs-p-maj-2} we begin a discussion of some of these interconnections.
Next we exhibit a relationship between strong majorization $\preccurlyeq$ (recall Definition~\ref{def:strong-majorization}) and $\infty$-majorization.
As stated, it may seem to apply to summable sequences, but in fact the hypotheses, majorization and not strong majorization, negate that possibility as addressed just after the proof of Proposition~\ref{prop:strong-maj-vs-p-maj}.
\begin{proposition}
  \label{prop:strong-maj-vs-p-maj}
  If $\xi,\eta\in c_0^+$ and $\xi\prec\eta$, then 
\[ \xi\not\preccurlyeq\eta  \implies \xi\prec_\infty\eta. \] 
\end{proposition}

\begin{proof}
  It suffices to show that if $\xi\prec\eta$, then the $\liminf$ condition implies $\xi\prec_p \eta$ for every $p\in\mathbb{N}$.
Indeed, suppose that 
\[ \liminf_{n\to\infty}\sum_{k=1}^n \left(\eta^*_k-\xi^*_k\right)=\epsilon>0. \] 
Then since $\xi\in c_0^+$ one can choose $N\in\mathbb{N}$ for which $\xi^*_k<\frac{\epsilon}{2p}$ for all $k\ge N$.
Then for all $n$ sufficiently large for which both $n\ge N$ and $\sum_{k=1}^n (\eta^*_k-\xi^*_k) > \frac{\epsilon}{2}$ one has that
\[ \sum_{k=1}^{n+p} \xi^*_k < \sum_{k=1}^n \xi^*_k + p\cdot\frac{\epsilon}{2p} < \sum_{k=1}^n \eta^*_k. \qedhere \]
\end{proof}

Note that Proposition~\ref{prop:strong-maj-vs-p-maj} applies only to $\xi,\eta\notin\ell^1$, since if either $\xi$ or $\eta$ is summable, majorization implies both are summable, and in this case majorization and strong majorization are equivalent (see Definition~\ref{def:strong-majorization}, succeeding comment).
Furthermore, the converse of Proposition~\ref{prop:strong-maj-vs-p-maj} fails because there exist sequences $\xi,\eta\in c_0^*\setminus\ell^1$ for which  $\xi\prec_\infty\eta$ hence also $\xi\prec\eta$, and yet $\xi\preccurlyeq\eta$, as exhibited in the following example.
\begin{example}
  \label{ex:half-two-ampliation}
  In this exposition so far we have not considered any hands-on
  examples of $\infty$-majorization. In refuting this possible converse of Proposition~\ref{prop:strong-maj-vs-p-maj} we provide one, but first we explain our natural motivation for it.
  
  \emph{Motivation}.
Suppose $A\in K(H)^+$ has $\trace R_A=\infty=\trace R_A^\perp$.
Then there is some basis with respect to which $A=\diag\angles{s_1(A),0,s_2(A),0,\ldots}$.
It is a natural question to ask if there exists some $B\in E(\mathcal{U}(A))$ with $\trace R_B^\perp=0$.
The answer is yes by Theorem~\ref{thm:sufficiencyofpmajorization}, but there is a more straightforward way to see this in this case.
Indeed, if we let $R_j$ denote the $2\times 2$ rotation by $\nicefrac{\pi}{4}$ acting on the subspace $\spans\{e_{2j-1},e_{2j}\}$, and $U=\bigoplus R_j$, then $B\coloneqq{}E(UAU^*)=\diag\frac{1}{2}D_2 s(A)$, where $D_2\phi$ denotes the $2$-ampliation of $\phi\in\ell^\infty$, i.e., $D_2\phi=\angles{\phi_1,\phi_1,\phi_2,\phi_2,\ldots}$. 

  \emph{Example}.
Fix any $\eta\in c_0^*$ and choose $\xi\coloneqq{}\frac{1}{2}D_2\eta\prec_\infty \eta$, where $\infty$-majorization is easily verified (hint: $N_p=p$).
Then
\[ \sum_{j=1}^k (\eta_j-\xi_j) =
\begin{cases}
  \sum_{j=n+1}^{2n} \eta_j &\text{if $k=2n$, $n\in\mathbb{N}$;} \\
  \sum_{j=n+1}^{2n-1} \eta_j + \frac{1}{2}\eta_n  &\text{if $k=2n-1$, $n\in\mathbb{N}$.} \\
\end{cases}
\]
From this it follows that 
\begin{equation}
  \label{eq:14}
  \left\lfloor \frac{k}{2} \right\rfloor \eta_k \le \sum_{j=1}^k (\eta_j-\xi_j) \le \left\lceil \frac{k}{2} \right\rceil \eta_{\lceil \nicefrac{k}{2} \rceil}. 
\end{equation}
From the second inequality in (\ref{eq:14}) it follows that  if $\liminf k \eta_k=0$, then $\xi=\frac{1}{2}D_2\eta\preccurlyeq\eta$ (e.g. for $\eta=\angles{((k+1)\log(k+1))^{-1}}$, $\lim k \eta_k=0$).
However, the first inequality of (\ref{eq:14}) shows the inverse, that if $\liminf k \eta_k>0$, then $\xi=\frac{1}{2}D_2\eta\not\preccurlyeq\eta$ (e.g. for $\eta=\angles{k^{-1}}$, $\lim k \eta_{2k}=\frac{1}{2}$).
The first example provides the failure of the converse of Proposition~\ref{prop:strong-maj-vs-p-maj}.
Moreover one has, for $\eta\in c_0^*$, $\frac{1}{2}D_2\eta\preccurlyeq\eta$ if and only if $\liminf k\eta_k=0$. 
\end{example}

The previous example shows that for appropriate $\eta\in c_0^*$, there exist $\xi\in c_0^*$ which are counterexamples to the converse of Proposition~\ref{prop:strong-maj-vs-p-maj}.
However, with more work one can show for \emph{every} sequence $\eta\in c_0^+$, there is some $\xi\in c_0^*$ with $\xi\prec_\infty\eta$ and $\xi\preccurlyeq\eta$, as the next proposition shows.
Though it will not be used later in this paper, we present it here for completeness.

\begin{proposition}
  \label{prop:strong-maj-vs-p-maj-2}
  For every $\eta\in c_0^+$ there exists some $\xi\in c_0^*$ with $\xi\prec_\infty\eta$ and $\xi\preccurlyeq\eta$.
\end{proposition}

\begin{proof}
  Without loss of generality, we may assume $\eta\in c_0^*$ since the definitions of $\infty$-majorization and strong majorization depend only on its monotonization $\eta^*$. 

  If $\liminf k \eta_k =0$, set $\xi=\frac{1}{2}D_2\eta$.
By the example preceding this proposition, we have $\xi\prec_\infty\eta$ and $\xi\preccurlyeq\eta$. 

  It therefore suffices to assume $\liminf k \eta_k>0$.
For this we adopt the following conventions.
Let juxtaposition of finite sequences denote the concatenation of those sequences.
Let $\ell(\phi)$ denote the length of a finite sequence $\phi$.
Let $\sum\phi$ be an abbreviation of $\sum_{j=1}^{\ell(\phi)}\phi_j$ and $\sum_a^b \phi$ an abbreviation for $\sum_{j=a}^b \phi_j$.
Finally, if $\phi$ is a sequence (not necessarily finite) and $a,b\in\dom\phi$ with $a\le b$, let $\phi|_a^b$ denote the finite subsequence $\angles{\phi_a,\phi_{a+1},\ldots,\phi_b}$.

  If $\liminf k \eta_k>0$ and $\eta\in c_0^*$, then $\eta>0$.
To construct $\xi$ we proceed inductively.
Let $\angles{\epsilon_k}$ be any positive sequence converging to zero and let $N'_0=1$.
Then let $N_1$ be the smallest positive integer for which $0<\eta_{N_1}\le \frac{\eta_{N'_0}}{2}$.
Let $p_1$ be the largest positive integer for which $p_1\eta_{N_1}\le \eta_{N'_0}$, i.e., $p_1\coloneqq{}\left\lfloor \frac{\eta_{N'_0}}{\eta_{N_1}}\right\rfloor\ge 2$ and $(p_1+1)\eta_{N_1}>\eta_{N'_0}=\eta_1$.
Then setting
\[ \xi^{(1)}=\eta|_2^{N_1}\overset{p_1}{\overbrace{\angles{\eta_{N_1},\ldots,\eta_{N_1}}}}, \]
this choice of $p_1$ guarantees that 
\begin{equation}
  \label{eq:28}
  0\le \sum_1^{N_1} \eta - \sum \xi^{(1)} = \eta_1 - p_1 \eta_{N_1} < \eta_{N_1}.
\end{equation}
Denote the difference between the lengths of $\xi^{(1)}$ and $\angles{\eta_j}_1^{N_1}$ by 
\[ M_1\coloneqq{}\ell(\xi^{(1)})-N_1=p_1-1\ge 1 \]
and then exploiting $\eta\to 0$, choose $N'_1>\ell(\xi^{(1)})$ for which $\sum\eta\big|_{N'_1-M_1}^{N'_1-1}<\epsilon_1$.
Then one has
\begin{equation}
  \label{eq:29}
  0\le \sum_1^{N'_1-1}\eta-\sum_1^{N'_1-1}\xi^{(1)}\eta\big\vert_{N_1+1}^{N'_1-1} < \eta_{N_1} + \epsilon_1 
\end{equation}
because, noting that the length of $\xi^{(1)}\eta\big\vert_{N_1+1}^{N'_1-1}$ is greater by $M_1$ than $N'_1-1$, one has
\begin{equation}
  \label{eq:25}
  \begin{aligned}
    0\le \sum_1^{N'_1-1}\eta-\sum_1^{N'_1-1}\xi^{(1)}\eta\big\vert_{N_1+1}^{N'_1-1} 
    &= \sum_1^{N'_1-1}\eta-\sum\left(\xi^{(1)}\eta\big\vert_{N_1+1}^{N'_1-1}\right)+\sum\eta\big\vert_{N'_1-M_1}^{N'_1-1} \\
    &= \sum_1^{N_1}\eta-\sum\xi^{(1)} + \sum\eta\big\vert_{N'_1-M_1}^{N'_1-1} \\
    &< \eta_{N_1} + \epsilon_1.
  \end{aligned}
\end{equation}

Continuing with the induction, suppose that for some $k\in\mathbb{N}$ as in the previous $k=1$ case we are given a finite decreasing sequence $\xi^{(k)}$ and $N_k<N_k+M_k=\ell(\xi^{(k)})<N'_k$, with $M_k\ge k$ and with the last term of $\xi^{(k)}$ being equal to $\eta_{N_k}$ and $\xi^{(k)}$ satisfying 
\begin{equation}
  \label{eq:19}
  0\le \sum_1^{N_k} \eta - \sum \xi^{(k)} <\eta_{N_k} \quad\text{and}\quad \sum \eta\big\vert_{N'_k-M_k}^{N'_k-1}<\epsilon_k.
\end{equation}
As in the previous $k=1$ case, by an argument identical to that of (\ref{eq:25}), together these (\ref{eq:19}) inequalities imply
\begin{equation}
  \label{eq:24}
  0\le \sum_1^{N'_k-1}\eta-\sum_1^{N'_k-1}\xi^{(k)} \eta\big\vert_{N_k+1}^{N'_k-1} < \eta_{N_k}+\epsilon_k.
\end{equation}

We then construct an extension $\xi^{(k+1)}$ of $\xi^{(k)}$ and $N_{k+1},M_{k+1},N'_{k+1}$ with $N'_k<N_{k+1}$ satisfying all of the properties in the preceding paragraph for $k+1$ replacing $k$ in all instances.
The procedure mimics the base case as follows.
Let $N_{k+1}$ be the smallest integer greater than $N'_k$ for which $0<\eta_{N_{k+1}}\le\frac{\eta_{N'_k}}{2}$.
Along with (\ref{eq:19}) this inequality implies that
\[ \sum \xi^{(k)} + 2\eta_{N_{k+1}} \le \sum_1^{N_k}\eta +
  \eta_{N'_k}. \]
Hence letting $p_{k+1}$ be the largest positive integer for which 
\begin{equation}
  \label{eq:30}
  \sum \xi^{(k)} + p_{k+1} \eta_{N_{k+1}} \le \sum_1^{N_k}\eta +
  \eta_{N'_k},
\end{equation}
one has $p_{k+1}\ge 2$.
And by the maximality of $p_{k+1}$ one also has 
\begin{equation}
  \label{eq:33}
  \sum \xi^{(k)} + (p_{k+1}+1)\eta_{N_{k+1}} > \sum_1^{N_k}\eta +
  \eta_{N'_k}.
\end{equation}
Adding to both sides of (\ref{eq:30}) and (\ref{eq:33}) the $\eta$-terms from $N_k+1$ to $N_{k+1}$ excluding $N'_k$, one defines $\xi^{(k+1)}$ as denoted and obtains
\begin{equation}
  \label{eq:31}
  \sum \underset{\xi^{(k+1)}}{\underbrace{\xi^{(k)} \eta\big\vert_{N_k+1}^{N'_k-1} \eta\big\vert_{N'_k+1}^{N_{k+1}} \overset{p_{k+1}\ \text{times}}{\overbrace{\angles{\eta_{N_{k+1}},\ldots,\eta_{N_{k+1}}}}}}} \le  \sum_1^{N_{k+1}}\eta,
\end{equation}
and 
\begin{equation}
  \label{eq:32}
  \sum \xi^{(k+1)} + \eta_{N_{k+1}} > \sum_1^{N_{k+1}} \eta.
\end{equation}
So from (\ref{eq:31})--(\ref{eq:32}) the difference of the (\ref{eq:32}) sums is nonnegative and less than $\eta_{N_{k+1}}$.
This shows that $\xi^{(k+1)}$, as defined in (\ref{eq:31}), satisfies the first inequalities of (\ref{eq:19}).
Note further that $\xi^{(k+1)}$ is decreasing since $\xi^{(k)}$ and $\eta$ are decreasing and because the last term of $\xi^{(k)}$ is $\eta_{N_k}$. 

As for $M_1$, and recalling that $p_{k+1}\ge 2$, set
\begin{align*}
  M_{k+1}
  &{\coloneqq}\  \ell(\xi^{(k+1)})-N_{k+1}=(\ell(\xi^{(k)})+N_{k+1}-(N_k+1)+p_{k+1})-N_{k+1}
  \\ 
  &= M_k-1+p_{k+1}> M_k\ge k,
\end{align*}
and hence $M_{k+1}\ge k+1$.
Next, since $\eta\in c_0^*$ we may choose some $N'_{k+1}>\ell(\xi^{(k+1)})$ satisfying the last inequality of (\ref{eq:19}), for $k+1$ replacing $k$.
These facts again imply (\ref{eq:24}) for $k+1$ replacing $k$ by an argument identical to (\ref{eq:25}).  

By induction, we have constructed $\xi^{(k)},N_k,M_k,N'_k$ with the desired properties (i.e., paragraph containing inequalities (\ref{eq:19}) and (\ref{eq:24})).
Furthermore, by construction each $\xi^{(k)}$ is an extension of the preceding ones.
Thus, the infinite sequence $\xi$ given by $\xi_j \coloneqq{} \xi^{(k)}_j$ when $1\le j\le \ell(\xi^{(k)})$ is well-defined.
Finally it suffices to show that $\xi\prec_\infty \eta$ and $\xi\preccurlyeq\eta$. 

In order to prove $\xi\prec_\infty\eta$ it suffices to observe the following two facts.
Firstly, if $1\le m\le N_1$ then 
\begin{equation}
  \label{eq:27}
  \sum_1^m \xi = \sum_1^m \eta\big\vert_2^{N_1}\angles{\eta_{N_1}} \le \sum_1^m \eta.
\end{equation}
Secondly, for each $k\in\mathbb{N}$, if $N_k<m\le N_{k+1}$ then
\begin{equation}
  \label{eq:26}
  \sum_1^{m+M_k} \xi \le \sum_1^m \eta
\end{equation}
because
\begin{align*}
  \sum_1^{m+M_k} \xi 
  &= \sum_1^{m+M_k} \xi^{(k)}\eta\big\vert_{N_k+1}^{N'_k-1} \eta\big\vert_{N'_k+1}^{N_{k+1}}\angles{\eta_{N_{k+1}}} \\
  &= \sum_1^{\ell(\xi^{(k)})} \xi^{(k)} + \sum_{\ell(\xi^{(k)})+1}^{m+M_k} \xi^{(k)}\eta\big\vert_{N_k+1}^{N'_k-1} \eta\big\vert_{N'_k+1}^{N_{k+1}}\angles{\eta_{N_{k+1}}} \\
  &= \sum_1^{\ell(\xi^{(k)})} \xi^{(k)} + \sum_1^{m-N_k} \eta\big\vert_{N_k+1}^{N'_k-1} \eta\big\vert_{N'_k+1}^{N_{k+1}}\angles{\eta_{N_{k+1}}} \\
  &\le \sum_1^{N_k} \eta + \sum_1^{m-N_k} \eta\big\vert_{N_k+1}^{N_{k+1}} = \sum_1^m \eta. 
\end{align*} 
Indeed, since $M_k\uparrow\infty$, these inequalities (\ref{eq:27}) and (\ref{eq:26}) imply $\xi\prec_k\eta$ for infinitely many $k\in\mathbb{N}$, i.e., $\xi\prec_\infty\eta$.

Finally, (\ref{eq:24}) implies $\xi\preccurlyeq\eta$ since
\[ \liminf_{k\to\infty} \sum_{j=1}^k (\eta_j-\xi_j) \le \liminf_{k\to\infty} \sum_{j=1}^{N'_k-1} (\eta_j-\xi_j) \le \liminf_{k\to\infty} (\eta_{N_k}+\epsilon_k) = 0. \qedhere \]
\end{proof}

\paragraph{Operator consequences} 
The following corollary of Theorem~\ref{thm:sufficiencyofpmajorization} and Proposition~\ref{prop:strong-maj-vs-p-maj} gives a method of ensuring membership in $E(\mathcal{U}(A))$, for $A\in K(H)^+$.
The purpose of this corollary is to provide a more easily computable way to make this determination in special cases.
For if one is given sequences $\xi,\eta\in c_0^+$, establishing $\xi\prec_p\eta$ or its negation seems more difficult than verifying $\xi\not\preccurlyeq\eta$, which requires only $\xi\prec\eta$ and the strict positivity of the associated $\liminf$ condition. 

\begin{corollary}\label{cor:sufficiencyofpositiveliminf}
Suppose $A,B\in K(H)^+$, $B\in \mathcal{D}$, $\trace R_B^\perp \le\trace R_A^\perp $, and $s(B)\prec s(A)$ but $s(B)\not\preccurlyeq s(A)$, then $B \in E(\mathcal{U}(A))$. 
\end{corollary}

\begin{proof}
  By Proposition~\ref{prop:strong-maj-vs-p-maj} one has $s(B)\prec_\infty s(A)$, and then using Theorem~\ref{thm:sufficiencyofpmajorization} one obtains $B\in E(\mathcal{U}(A))$. 
\end{proof}

\section{Approximate $p$-majorization (necessity)} \label{sec:appr-p-major}

Theorem~\ref{thm:sufficiencyofpmajorization} of the last section shows that if $p\coloneqq{}\trace R_A^\perp -\trace R_B^\perp\ge 0$ or when undefined we set $p=0$, then $p$-majorization ($s(B)\prec_p s(A)$) is a sufficient condition for membership in $E(\mathcal{U}(A))$, but Example~\ref{ex:nonnecessity-of-p-majorization} shows it is not necessary for $0<p<\infty$. 
In our quest to characterize $E(\mathcal{U}(A))$ in terms of sequence majorization, we introduce a new type of majorization called \emph{approximate $p$-majorization}, which is a necessary condition for membership in $E(\mathcal{U}(A))$.

\begin{definition}\label{def:approximatepmajorization}
Given $\xi,\eta\in c_0^+$ and $0\le p<\infty$, we say that $\xi$ is \emph{approximately $p$-majorized} by $\eta$, denoted $\xi\precsim_p\eta$, if $\xi\prec\eta$ and for every $\epsilon>0$, there exists an $N_{p,\epsilon}\in\mathbb{N}$ such that for all $n\ge N_{p,\epsilon}$, 
\[ \sum_{k=1}^{n+p} \xi_k \le  \sum_{k=1}^n \eta_k + \epsilon \eta_{n+1}.\]
Furthermore, if $\xi\precsim_p\eta$ for infinitely many $p\in\mathbb{N}$ (equivalently obviously, for all $p\in\mathbb{N}$), this we call approximate $\infty$-majorization and denote it by $\xi\precsim_\infty\eta$. 
\end{definition}

\begin{remark}\label{rem:relationshipbetweenpmajandapppmaj}
Notice from the above definition that if $\xi$ is $p$-majorized by $\eta$, then $\xi$ is trivially approximately $p$-majorized by $\eta$.
However, there is a partial converse with a small loss in that approximate $p$-majorization implies $(p-1)$-majorization.
That is, if $p>0$ and $\xi$ is approximately $p$-majorized by $\eta$, then by choosing $\epsilon<1$, from the above display
$\xi$ is $(p-1)$-majorized by $\eta$.
Combining these two facts yields that $\xi\prec_\infty\eta$ if and only if $\xi\precsim_\infty\eta$, which is a fact we will exploit later.
Furthermore, as we saw in the proof of Theorem~\ref{thm:sufficiencyofpmajorization} Case 1, if $\eta$ has only finitely many nonzero terms, then any $\xi$ which is majorized by $\eta$ is $\infty$-majorized by $\eta$ and so also approximately $\infty$-majorized by $\eta$.
\end{remark}

\begin{example}
  \label{ex:app-p-maj-not-p-maj}
  It is important to note that $p$-majorization is distinct from approximate $p$-majorization. 
That is, for each $0<p<\infty$ we exhibit sequences $\xi,\eta\in c_0^+$ with $\xi\precsim_p\eta$ but $\xi\not\prec_p\eta$. 
When $p=1$, it suffices to consider the sequences $\xi=\angles{\nicefrac{(2^{k+1}-3)}{2^{2k}}}$ and $\eta=\angles{2^{-k}}$. 
Elementary calculations verify that $\xi,\eta\in c_0^*$ (that is, $\xi=\xi^*$ and $\eta=\eta^*$), and $\xi\precsim_1\eta$ but $\xi\not\prec_1\eta$. 
To produce analogous sequences for any $p>1$, define
\[ \xi^{(p)} \coloneqq{} \langle\overbrace{1,\ldots,1}^{\text{$p-1$ times}},\xi_1,\xi_2,\ldots\rangle \quad\text{and}\quad \eta^{(p)}  \coloneqq{} \angles{p-1,\eta_1,\eta_2,\ldots}. \]
Then $\xi^{(p)}\precsim_p\eta^{(p)}$ but $\xi^{(p)}\not\prec_p\eta^{(p)}$, which proves that $p$-majorization and approximate $p$-majorization are distinct. 
However, it should be noted that these examples were not nearly as easy for us to come by as those for $p$-majorization. 
In particular the examples immediately preceding Remark~\ref{rem:strategy} came naturally, but the single example above took some effort. 
For further discussion on pairs of sequences $\xi\precsim_p\eta$ but $\xi\not\prec_p\eta$ see unifying Remark~\ref{rem:app-p-maj-via-orthostochastic}.
\end{example}

Our main theorem on necessity for membership in $E(\mathcal{U}(A))$ depends on approximate $p$-majorization:

\begin{theorem}\label{thm:necessityofapppmajorization}
  Suppose $A\in K(H)^+$ and $B\in E(\mathcal{U}(A))$.
If 
\[ p = \min \{ n\in\mathbb{N}\cup\{0,\infty\} \mid \trace R_A^\perp \le\trace R_B^\perp +n \}, \]
then $s(B) \precsim_p s(A)$. 
\end{theorem}

\begin{proof}
  Suppose $A,B$ and $p$ are as in the hypotheses of this theorem.
We may assume that $A$ has infinite rank for otherwise the conclusion holds because of Remark~\ref{rem:relationshipbetweenpmajandapppmaj} and Theorem~\ref{thm:sufficiencyofpmajorization} proof of Case 1.

  We may also assume that $N\coloneqq{}\trace R_B^\perp <\infty$.
Otherwise, $\trace R_B^\perp =\infty=\trace R_A^\perp $ by Proposition~\ref{prop:cantaddzeros} and therefore $p=0$.
Thus we would need to prove $s(B)\precsim_0 s(A)$ which is equivalent to $s(B)\prec s(A)$, and this holds by Theorem~\ref{thm:kwpartialisometryorbit}.
  
Note that since $\trace R_B^\perp<\infty$, $p$ satisfies $\trace R_A^\perp = \trace R_B^\perp + p = N+p$. 
It suffices to show $s(B)\precsim_r s(A)$ for all $r\in\mathbb{N}$ with $r\le p$. 
Because $\trace R_A^\perp =N+p\ge N+r$, without loss of generality via unitary equivalence for $A$ and permutations for $B$, and the fact that conjugation by a permutation commutes with the expectation $E$, we may assume that 
\[ A = \diag\langle\underset{N+r}{\underbrace{0,\ldots,0}},\tilde\eta\rangle = \diag\eta' \quad\text{and}\quad B = \diag\langle\underset{N}{\underbrace{0,\ldots,0}},s(B)\rangle = \diag\xi, \] 
where $\tilde\eta$ is $s(A)$ interspersed with $p-r$ zeros. (To aid intuition, in case $p<\infty$ we may choose $r=p$ and hence also $\tilde\eta=s(A)$.)
Since $B\in E(\mathcal{U}(A))$ by Lemma~\ref{lem:kwstochasticcontraction}(iv), there exists a unistochastic matrix $Q=(q'_{ij})$ for which $Q\eta' = \xi$.
However, only double stochasticity of $Q$ is used here.
For all $m\in\mathbb{N}$ one has
\begin{equation}
  \label{eq:1}
  \sum_{i=1}^m \xi_i = \sum_{i=1}^m (Q\eta')_i = \sum_{i=1}^m \sum_{j=1}^\infty q'_{ij}\eta'_j = \sum_{j=1}^\infty \eta'_j \sum_{i=1}^m q'_{ij}.
\end{equation}
For a doubly-stochastic matrix $Q = (q'_{ij})$, denote the last quantity in equation~\eqref{eq:1} as $f_m(Q,\eta')$.
It is clear that for fixed $\eta'$, $f_m(Q,\eta')$ depends solely on the columns of $Q$. 

Now fix any $0<\epsilon<1$, and choose $N+r<N_{r,\epsilon}\in\mathbb{N}$ for which
\begin{equation}
  \label{eq:2}
  N+r-\epsilon < \sum_{i=1}^{N_{r,\epsilon}} \sum_{j=1}^{N+r} q'_{ij}.
\end{equation}
The existence of $N_{r,\epsilon}$ follows from column-stochasticity, which yields 
\[ 
\sum_{i=1}^\infty \sum_{j=1}^{N+r} q'_{ij}=N+r. 
\]
Certainly inequality~(\ref{eq:2}) holds with $N_{r,\epsilon}$ replaced by any $m\ge N_{r,\epsilon}$, since $Q$ is a doubly stochastic matrix and so its entries are nonnegative.

Consider a permutation matrix $\Pi_m$ which fixes the first $N+r$ coordinates and has the property that 
\[ \Pi_m \eta' = \eta =  \langle\underset{N+r}{\underbrace{0,\ldots,0}},\underset{m-N-r+1}{\underbrace{s_1(A),s_2(A),\ldots,s_{m-N-r+1}(A)}},\eta_{m+2},\eta_{m+3},\ldots\rangle, \] 
where $\eta_{m+2},\eta_{m+3},\ldots$ are the remaining $\tilde\eta$-terms. 
(To aid intuition, when $p<\infty$ and we choose $r=p$ and $\tilde\eta=s(A)$, $\eta'$ is already in this form without need of the permutation $\Pi_m$.)
Then notice that $Q$ is both doubly stochastic and satisfies inequality~(\ref{eq:2}) if and only if $Q\Pi_m^{-1}$ does the same.
Notice also that inequality~(\ref{eq:2}) depends only on the first $N+r$ columns of $Q$.
Then direct computations show
\begin{equation}
\label{eq:37}
f_m(Q,\eta') = f_m(Q\Pi_m^{-1} \Pi_m,\eta') = f_m(Q\Pi_m^{-1},\Pi_m\eta') = f_m(Q\Pi_m^{-1}, \eta).
\end{equation}

Denote the entries of $Q\Pi_m^{-1}$ by $q_{ij}$. 
From the definition of $\eta$ above it is clear that $\eta_{m+1}\ge \eta_k$ whenever $k\ge m+1$. 
This yields an upper bound for $f_m(Q\Pi_m^{-1},\eta)$ when $m\ge N_{r,\epsilon}$
\begin{align*}
  f_m(Q\Pi_m^{-1},\eta) &= \sum_{j=1}^\infty \eta_j \sum_{i=1}^m q_{ij}  \\ 
  &= \sum_{j=N+r+1}^m \eta_j \sum_{i=1}^m q_{ij} + \sum_{j=m+1}^\infty \eta_j \sum_{i=1}^m q_{ij}  \\
  &\le \sum_{j=N+r+1}^m \eta_j \sum_{i=1}^m q_{ij} + \Big(\sup_{k>m}\eta_k\Big) \sum_{j=m+1}^\infty \sum_{i=1}^m q_{ij}  \\
  &= \sum_{j=N+r+1}^m \eta_j \sum_{i=1}^m q_{ij} + \eta_{m+1} \left( m - \sum_{j=1}^m \sum_{i=1}^m q_{ij} \right)  \\
  &= \sum_{j=N+r+1}^m \eta_j \sum_{i=1}^m q_{ij} + \eta_{m+1} \left( \sum_{j=1}^m \left( 1-\sum_{i=1}^m q_{ij} \right) \right)  \\
  &= \sum_{j=N+r+1}^m \eta_j \sum_{i=1}^m q_{ij} +   \left( \sum_{j=N+r+1}^m\eta_{m+1} \left( 1-\sum_{i=1}^m q_{ij} \right) \right)  \\
  &\qquad + \left( \sum_{j=1}^{N+r}\eta_{m+1} \left( 1-\sum_{i=1}^m q_{ij} \right) \right) \nonumber \\
  &\le \sum_{j=N+r+1}^m \eta_j \sum_{i=1}^m q_{ij} + \sum_{j=N+r+1}^m \eta_j \left( 1-\sum_{i=1}^m q_{ij} \right)  \\
  &\qquad + \eta_{m+1}\left( N+r - \sum_{j=1}^{N+r} \sum_{i=1}^{m} q_{ij} \right) \nonumber \\
  &<  \sum_{j=N+r+1}^m \eta_j + \epsilon\eta_{m+1} \quad \text{by (\ref{eq:2}).}  
\end{align*}
The inequalities above along with equation \eqref{eq:37} yield
\begin{align*} 
  \sum_{i=1}^{m-N} s_i(B) = \sum_{i=1}^m \xi_i &= f_m(Q,\eta') = f_m(Q\Pi_m^{-1},\eta) \\
  &\le \sum_{j=N+r+1}^m \eta_j + \epsilon \eta_{m+1} \\
  &=  \sum_{i=1}^{m-N-r} s_i(A) + \epsilon s_{m-N-r+1}(A).
\end{align*}
Since $\epsilon$ is arbitrary, $s(B) \precsim_r s(A)$, and since the positive integer $r\le p$ is arbitrary, $s(B) \precsim_p s(A)$. 
\end{proof}

One of our main results is Corollary~\ref{cor:conjecturetrueforinfinitemo} which, in the rather general setting where $A$ has infinite rank and infinite dimensional kernel, we obtain a precise characterization of $E(\mathcal{U}(A))$ in terms of majorization and $\infty$-majorization.

\begin{corollary}\label{cor:conjecturetrueforinfinitemo}
  Suppose $A\in K(H)^+$ has infinite rank and infinite dimensional kernel $(\trace R_A=\infty=\trace R_A^\perp)$.
Then
\[ E(\mathcal{U}(A)) = E(\mathcal{U}(A))_{fk}\sqcup E(\mathcal{U}(A))_{ik}, \]
the members of $E(\mathcal{U}(A))$ with finite dimensional kernel and infinite dimensional kernel, respectively, are characterized by   
\[ E(\mathcal{U}(A))_{fk}=\{ B\in\mathcal{D}\cap K(H)^+\mid s(B)\prec_\infty s(A)\quad\text{and}\quad \trace R_B^\perp<\infty\} \]
and
\[ E(\mathcal{U}(A))_{ik}=\{ B\in\mathcal{D}\cap K(H)^+ \mid s(B)\prec s(A)\quad\text{and}\quad \trace R_B^\perp=\infty\}. \] 
\end{corollary}

\begin{proof}
  If $B\in E(\mathcal{U}(A))$, then $B\in E(\mathcal{V}(A))$ and we know that $s(B)\prec s(A)$ by Theorem~\ref{thm:kwpartialisometryorbit}.
But when $\trace R_B^\perp<\infty$, from Theorem~\ref{thm:necessityofapppmajorization} we know that $s(B)\precsim_\infty s(A)$ which is equivalent to $s(B)\prec_\infty s(A)$.
Thus the left-hand set in Corollary~\ref{cor:conjecturetrueforinfinitemo} is contained in the right-hand set.

Next suppose that $B\in \mathcal{D}\cap K(H)^+$ lies in the right-hand set.
If $\trace R_B^\perp =\infty$, then Theorem~\ref{thm:sufficiencyofpmajorization} with $p=0$ shows that $B\in E(\mathcal{U}(A))$.
Similarly, if $\trace R_B^\perp <\infty$, then $s(B)\prec_\infty s(A)$ and again, by  Theorem~\ref{thm:sufficiencyofpmajorization} for $p=\infty$, we find that $B\in E(\mathcal{U}(A))$. 
\end{proof}

Corollary~\ref{cor:conjecturetrueforinfinitemo} can be expressed as 
\[ E(\mathcal{U}(A)) = \]
\[ \left\{ B\in\mathcal{D}\cap K(H)^+ \big\vert \exists\ 0\le p\le\infty\text{ for which }\trace R_B^\perp+p=\infty\text{ and }  s(B)\precsim_p s(A) \right\}. \]
This motivates the following conjectured characterization for $E(\mathcal{U}(A))$, which remains an open problem.
And if this conjecture should prove false, is there a proper majorization characterization of $E(\mathcal{U}(A))$?

\begin{conjecture}
  \label{conj:app-p-maj-schur-horn}
  Let $A\in K(H)^+$.
  Then
  \[ E(\mathcal{U}(A)) = \]
  \[ \bigcup_{0\le p\le \infty} \left\{ B\in\mathcal{D}\cap K(H)^+ \big\vert s(B)\precsim_p s(A)\text{ and } \trace R_B^\perp \le \trace R_A^\perp \le \trace R_B^\perp +p \right\} \]
\end{conjecture}

The following remark provides a method for producing pairs of sequences $\xi\in c_0^+$ and $\eta\in c_0^*$ for which $\xi\precsim_p\eta$ but $\xi\not\prec_p\eta$. 

\begin{remark}
  \label{rem:app-p-maj-via-orthostochastic}
  Recall that Example~\ref{ex:nonnecessity-of-p-majorization} provided an orthostochastic matrix $Q$ such that for every $\eta\in c_0^*$, setting $\xi = Q\tilde\eta$, where $\tilde\eta=\angles{0,\eta}$, yields $\xi\not\prec_1\eta$. 
Using Lemma~\ref{lem:kwstochasticcontraction} one has $\diag\xi\in E(\mathcal{U}(\diag\eta))$.
By Theorem~\ref{thm:necessityofapppmajorization}, 
\[ \xi^* = s(\diag\xi) \precsim_1 s(\diag\tilde\eta) = \tilde\eta^*=\eta, \]
hence $\xi \precsim_1 \eta$. 
In general, given two finite sequences, say $\phi,\zeta$, of lengths $n$ and $n+p$ (where $n$ is arbitrary) and having the same sum (i.e., $\sum_{j=1}^n \phi_j = \sum_{j=1}^{n+p}\zeta_j$), one can prepend $\phi$ to $\eta$ and $\zeta$ to $\xi$, clearly obtaining new sequences $\xi',\eta'$ with $\xi' \precsim_p \eta'$ but $\xi' \not\prec_p \eta'$. 

Although we did not mention it earlier, in Example~\ref{ex:nonnecessity-of-p-majorization} there is a substantial amount of freedom in choosing $Q$. 
In particular, examination of \cite[Example~6.11]{KW-2010-JoFA} ensures that each sum-1 strictly positive column vector followed by a Gram--Schmidt process produces a distinct orthostochastic matrix $Q$ that can be used in Example~\ref{ex:nonnecessity-of-p-majorization}.
Perhaps these orthostochastic $Q$ can be exploited or modified to prove Conjecture~\ref{conj:app-p-maj-schur-horn}. 
\end{remark}

\section{$E(\mathcal{U}(A))$ convexity}
\label{sec:conv-expect-unit}

Historically, convexity played a central role and is ubiquitous in majorization theory. For example, Horn \cite[Theorem 1]{Hor-1954-AJoM} integrates theorems of Hardy, Littlewood and P\'{o}lya \cite{Hardy.etal-1988} and Birkhoff \cite{Bir-1946-UNTRA} to prove
\[ \{\xi\in\mathbb{R}^n \mid \xi\prec\eta\} = \convex \{\tilde\eta\in\mathbb{R}^n \mid \tilde\eta_k = \eta_{\pi(k)}, \pi\in\Pi_n\}, \]
where $\Pi_n$ is the set of $n\times n$ permutation matrices. 
For operators, in \cite{Hor-1954-AJoM}, using \cite{Sch-1923-SBMG}, Horn proved that $E(\mathcal{U}(X))$ is convex whenever $X=X^*\in M_n(\mathbb{C})$ by establishing the characterization
\[ E(\mathcal{U}(X)) = \{ \diag d  \mid d\in\mathbb{R}^n, d\prec\lambda\}, \]
where $\lambda$ is the eigenvalue sequence of $X$ (Theorem~\ref{thm:schur-horn}).
However, the verification that $E(\mathcal{U}(X))$ is convex is immediate from its majorization characterization even without the theorem of Horn integrating Hardy, Littlewood, P\'olya and Birkhoff mentioned above.
Likewise it is straightforward to verify that if $\eta\in c_0^+$, then
\[ \{ \xi\in c_0^+ \mid \xi\prec\eta \} \]
is convex.
In particular, this leads to the results of Kaftal and Weiss on the convexity of the expectation of the partial isometry orbit of a positive compact operator.

\begin{theorem}[\protect{\cite[Corollary 6.7]{KW-2010-JoFA}}]\label{thm:convexityofexpectationofpartialisometryorbit}
  Let $A\in K(H)^+$.
Then 
  \begin{itemize}[itemsep=0pt]
  \item $E(\mathcal{V}(A))$ is convex.
  \item If $R_A=I$ or $A$ has finite rank, then $E(\mathcal{U}(A))$ is convex. 
  \end{itemize}
\end{theorem}

Since we have a characterization of $E(\mathcal{U}(A))$ when $A\in K(H)^+$ has both infinite rank and infinite dimensional kernel, it seems natural to ask if $E(\mathcal{U}(A))$ is convex in this case.
The answer is positive, however the verification is much less obvious to us (see below Corollary~\ref{cor:convexityofexpectationofunitaryorbit}). But first, a lemma.

\begin{lemma}\label{lem:convexityofpqsequences}
  Suppose that $\xi,\zeta,\eta\in c_0^+$, $0<\lambda<1$, $0\le p,q\le\infty$ such that $\xi\precsim_p\eta$, $\zeta\precsim_q\eta$.
If $r=\min\{ p+\abs{\xi^{-1}(0)\setminus\zeta^{-1}(0)}, q+\abs{\zeta^{-1}(0)\setminus\xi^{-1}(0)}\}$, then $\lambda\xi+(1-\lambda)\zeta\precsim_r \eta$.
\end{lemma}

\begin{proof}
  Set $\phi\coloneqq{}\lambda\xi+(1-\lambda)\zeta$. 
There are two cases: either $\eta$ has finite support, or not. 
If the former, then since $\xi\prec\eta$ and $\zeta\prec\eta$, one easily has $\phi\prec\eta$ by the comment immediately preceding Theorem~\ref{thm:convexityofexpectationofpartialisometryorbit}. 
Then since $\eta$ has finite support, we can improve this to $\phi\prec_\infty\eta$ (see Remark~\ref{rem:relationshipbetweenpmajandapppmaj} or proof of Theorem~\ref{thm:sufficiencyofpmajorization} Case 1), which is equivalent to $\phi\precsim_\infty\eta$. 
Thus $\phi\precsim_r\eta$.

The second case: $\eta$ has infinite support. 
For now, suppose both $p,q$ are finite.
Let $\pi:\mathbb{N}\to\mathbb{N}\setminus\phi^{-1}(0)$ be a bijection monotonizing $\phi$: $\phi^*_k = \phi_{\pi(k)}$.
Since $\phi^{-1}(0)=\xi^{-1}(0)\cap\zeta^{-1}(0)$, one has
\begin{equation}
  \label{eq:34}
  \mathbb{N}\setminus\phi^{-1}(0) = \left(\mathbb{N}\setminus(\xi^{-1}(0)\cup\zeta^{-1}(0))\right) \sqcup \left(\xi^{-1}(0)\setminus\zeta^{-1}(0)\right) \sqcup \left(\zeta^{-1}(0)\setminus\xi^{-1}(0)\right),
\end{equation}
which immediately yields the disjoint union
\begin{align*}
  \pi([1,m]) 
&=\, \left[\pi([1,m])\cap\left(\mathbb{N}\setminus(\xi^{-1}(0)\cup\zeta^{-1}(0))\right)\right] \qquad \text{(cardinality $k_m$)} \\
&\qquad \quad \sqcup \left[\pi([1,m])\cap\left(\xi^{-1}(0)\setminus\zeta^{-1}(0)\right)\right] \qquad \text{(cardinality $s_m$)}\\
&\qquad  \quad \sqcup \left[\pi([1,m])\cap\left(\zeta^{-1}(0)\setminus\xi^{-1}(0)\right)\right] \qquad \text{(cardinality $t_m$)}
\end{align*}
for each $m\in\mathbb{N}$.
Therefore $m = k_m + s_m +t_m$, and each term increases to the cardinality of its corresponding set from equation~(\ref{eq:34}). 
From this cardinality equation, it is clear that $m-s_m=k_m+t_m$ and $m-t_m=k_m+s_m$ are both increasing sequences. 
However, we may further conclude that they increase without bound. 
Indeed, 
\begin{align*}
  m-s_m &= \abs{\pi([1,m])\cap\left(\mathbb{N}\setminus(\xi^{-1}(0)\cup\zeta^{-1}(0))\right)}\\ 
  &\qquad + \abs{\pi([1,m])\cap\left(\xi^{-1}(0)\setminus\zeta^{-1}(0)\right)} \\
  &= \abs{\pi([1,m])\cap\left(\mathbb{N}\setminus\zeta^{-1}(0)\right)}, 
\end{align*}
which shows $m-s_m\uparrow\abs{\mathbb{N}\setminus\zeta^{-1}(0)}=\abs{\supp\zeta}$.
Likewise $m-t_m\uparrow\abs{\supp\xi}$. 
And $\xi,\zeta$ are both infinitely supported since they are majorized by $\eta$ and $\eta$ is infinitely supported (as shown simply in the proof of Theorem~\ref{thm:sufficiencyofpmajorization} Case 2, first paragraph).
Thus we have verified $m-s_m,m-t_m\uparrow\infty$.
Therefore, given $0<\epsilon<1$, once $m$ is sufficiently large so that $m-s_m\ge N_{p,\epsilon}$ and $m-t_m\ge N_{q,\epsilon}$, one has
\begin{align*}
  \sum_{k=1}^m \phi^*_k &= \lambda \sum_{k=1}^m \xi_{\pi(k)} + (1-\lambda) \sum_{k=1}^m \zeta_{\pi(k)} \\
  &\le \lambda\sum_{k=1}^{m-s_m} \xi^*_k + (1-\lambda)\sum_{k=1}^{m-t_m} \zeta^*_k \\
  &\le \lambda\left(\sum_{k=1}^{m-s_m-p} \eta^*_k + \epsilon \eta_{m-s_m-p+1}\right) \\
  &\quad + (1-\lambda)\left(\sum_{k=1}^{m-t_m-q}\eta^*_k + \epsilon \eta^*_{m-t_m-q+1}\right) \\
  &\le \sum_{k=1}^{m-r_m} \eta^*_k + \epsilon \eta^*_{m-r_m+1}, 
\end{align*}
where $r_m = \min\{ s_m+p , t_m + q \}$.
The above computation proves that $\phi\precsim_{r_m}\eta$.
But since $r_m\uparrow r$, either $r_m\uparrow\infty$, or eventually $r_m=r$. In either case $\phi\precsim_r\eta$. 

Finally, to remove the restriction that $p,q$ be finite, observe that the above proof actually showed that for any $p'\le p$, $q'\le q$ with $p',q'<\infty$, one has $\phi\precsim_{r'_m}\eta$, where $r'_m = \min\{s_m+p',t_m+q'\}$. The proof now splits into several subcases.

If $p<\infty$, $q=\infty$ and $\abs{\xi^{-1}(0)\setminus\zeta^{-1}(0)}<\infty$, then if one chooses both $p'=p$ and $q'=p+\abs{\xi^{-1}(0)\setminus\zeta^{-1}(0)}$, one has $r'_m=r=p+\abs{\xi^{-1}(0)\setminus\zeta^{-1}(0)}$ for $m$ sufficiently large so that $s_m = \abs{\xi^{-1}(0)\setminus\zeta^{-1}(0)}$.

If $p<\infty$, $q=\infty$ and $\abs{\xi^{-1}(0)\setminus\zeta^{-1}(0)}=\infty$, then for $p'=p$ and any $q'<\infty$, eventually $r'_m$ reaches $q'$ since $s_m\uparrow\infty$. Therefore $\phi\precsim_{q'}\eta$, which means $\phi\precsim_\infty\eta$ since $q'$ was arbitrary.

The cases where $p=\infty$ and $q<\infty$ hold by symmetric arguments.

If $p=\infty=q$, then for $p'=q'=k\in\mathbb{N}$, one has $r'_m\ge k$ and so $\phi\precsim_k\eta$, hence $\phi\precsim_\infty\eta$. 
\end{proof}

Examination of the proof of Lemma~\ref{lem:convexityofpqsequences} actually shows that we may replace approximate $p$-majorization with $p$-majorization everywhere in the statement of the lemma and the result remains valid.
Indeed, the only difference in the proof is that the terms involving $\epsilon$ disappear when $p$-majorization is used. 

An operator $E(\mathcal{U}(A))$ consequence of this is:

\begin{corollary}\label{cor:convexityofexpectationofunitaryorbit}
  If $A\in K(H)^+$ and $\trace R_A=\infty=\trace R_A^\perp $, then $E(\mathcal{U}(A))$ is convex. 
\end{corollary}

\begin{proof}
  Take $B,C\in E(\mathcal{U}(A))$ and $0<\lambda<1$.
Let $D=\lambda B+(1-\lambda) C$ and also let $d=\lambda b +(1-\lambda) c$ be their corresponding diagonal sequences in $c_0^+$.
If $\trace R_D^\perp =\infty$, that $D\in E(\mathcal{U}(A))$ follows immediately from Corollary~\ref{cor:conjecturetrueforinfinitemo} and from the fact that convex combinations of elements majorized by $\eta$ are themselves majorized by $\eta$.
However, if $\trace R_D^\perp <\infty$, we use the dichotomy of Corollary~\ref{cor:conjecturetrueforinfinitemo} for $B,C$. 
Observing that $\trace R_B^\perp  = \abs{b^{-1}(0)}$ and similarly for $C,D$, simply notice that either $\abs{b^{-1}(0)\setminus d^{-1}(0)}=\infty$ or $\abs{b^{-1}(0)\setminus d^{-1}(0)}<\infty$ (so $\trace R_B^\perp<\infty$), in which case $b\prec_\infty a$ by Corollary~\ref{cor:conjecturetrueforinfinitemo}.
Likewise, $\abs{c^{-1}(0)\setminus d^{-1}(0)}=\infty$ or $c\prec_\infty a$.
Thus by Lemma~\ref{lem:convexityofpqsequences} one has $d\prec_\infty a$.
And so by Corollary~\ref{cor:conjecturetrueforinfinitemo}, $D\in E(\mathcal{U}(A))$. 
\end{proof}

\bibliography{References}{}
\bibliographystyle{elsarticle-num}

\end{document}